\documentclass[11pt,a4paper]{article}
\usepackage{graphicx}
\usepackage{pdfsync}
\usepackage{amsmath,amsthm,amsfonts}
\usepackage{color}

\newcommand\R{\mathbb{R}}
\newcommand\Z{\mathbb{Z}}
\newcommand\N{\mathbb{N}}
\newcommand\mass{{\bf M}}
\newcommand\calH{\mathcal{H}}
\newcommand\calM{\mathcal{M}}
\newcommand\calR{\mathcal{R}}
\newcommand\Ext{\mathcal{E}}
\newcommand\calP{\mathcal{P}}
\newcommand\Snmu{\mathcal{S}^{n-1}}
\newcommand\Suno{\mathcal{S}^{1}}
\newcommand\MBCC{{\mathcal M}_{\rm df}^{(m)}}
\newcommand\MBCCOmega[1][\Omega]{{\mathcal M}_{\rm df}^{(m)}(#1)}
\newcommand{\weakstarto}{{\ensuremath{\overset{\ast}{\rightharpoonup}}}\,}
\newcommand\e{\varepsilon}
\newcommand\eps{\varepsilon}
\newcommand\supp{\mathop{\mathrm{supp\,}}}
\newcommand\dist{{\mathrm{dist}}}
\newcommand\sgn{{\mathrm{sgn}}}
\newcommand\Lip{{\mathrm{Lip}}}
\newcommand\Id{{\mathrm{Id}}}
\newcommand\psie{\psi_\mathrm{e}}
\newcommand\cc{c}
\newcommand\diam{\mathop{\mathrm{diam}}}
\newcommand{\LM}[1]{\hbox{\vrule width.2pt \vbox to#1pt{\vfill \hrule width#1pt height.2pt}}}
\newcommand{\LL}{{\mathchoice{\,\LM7\,}{\,\LM7\,}{\,\LM5\,}{\,\LM{3.35}\,}}}
\newcommand\Tan{\mathrm{Tan}}
\def\nabla{D}
\newtheorem{theorem}{Theorem}[section]
\newtheorem{lemma}[theorem]{Lemma}
\newtheorem{corollary}[theorem]{Corollary}
\newtheorem{proposition}[theorem]{Proposition}
\numberwithin{equation}{section}
\newcounter{Nummer}
\newenvironment{steplist}{
\setcounter{Nummer}{0}
\begin{list}{{\em Step \arabic{Nummer}:}}
{\setlength\leftmargin{0pt}
\setlength\labelwidth{0cm}
\setlength\itemindent{.2cm}
\usecounter{Nummer}} 
}{
\end{list}
}
\begin{document}
\begin{center}
{ \LARGE
Modeling of dislocations and relaxation of functionals on 1-currents with
discrete multiplicity
 \\[5mm]}
{September 22, 2014}\\[5mm]
Sergio Conti$^{1}$, Adriana Garroni$^{2}$, and 	
Annalisa Massaccesi$^3$\\[2mm]
{\em $^1$ Institut f\"ur Angewandte Mathematik,
Universit\"at Bonn\\ 53115 Bonn, Germany }\\
{\em $^{2}$ Dipartimento di Matematica, Sapienza, Universit\`a di Roma\\
00185 Roma, Italy}\\
[1mm]
{\em $^{3}$ Dipartimento di Matematica ``Federigo Enriques'',\\
 Universit\`a
  degli Studi di Milano, 
20133 Milano, Italy}
\\[3mm]
\begin{minipage}[c]{0.8\textwidth}\small
Abstract: In the modeling of dislocations one is lead naturally to energies concentrated
on lines, where the integrand depends on the orientation and on the Burgers
vector of the dislocation, which belongs to a discrete lattice.
The dislocations may be identified with divergence-free matrix-valued measures
supported on curves or with 1-currents with multiplicity in a lattice. In
this paper we develop the theory of relaxation for these energies and provide
one physically motivated example in which the relaxation for some Burgers
vectors is nontrivial and can be determined explicitly. From a
technical viewpoint the key ingredients are an approximation and a structure
theorem for 1-currents with multiplicity in a lattice.
\end{minipage}
\end{center}

\section{Introduction}
Dislocations are topological  singularities in crystals, which 
may be described by lines to which a lattice-valued vector, called Burgers
vector, is associated. They may be identified  with divergence-free
matrix-valued measures 
supported on curves or equivalently with 1-currents with multiplicity in a lattice and without boundary. 
The energetic modeling of dislocations leads naturally to energies
with linear growth concentrated 
on lines, where the integrand depends on the orientation and on the Burgers
vector of the dislocation.
The energy of a dislocation supported on a line $\gamma$, with tangent vector
$\tau:\gamma\to \Snmu $ and multiplicity $\theta:\gamma\to\Z^m$ takes the form

\begin{equation}\label{eq1}
\int_{\gamma}\psi(\theta,\tau) \,d\calH^1\,,
\end{equation}
restricted to the set of dislocation density tensors $\mu=\theta\otimes \tau
\calH^1\LL\gamma$ which are 
divergence-free, see for example \cite{HirthLothe1968,HullBacon}. In the
two-dimensional 
case such divergence-free measures can be identified with gradients of
characteristic functions in $BV$ and the problem can be treated as a
vector-valued partition problem \cite{AB1,AB2}; for a derivation of a
line-tension energy of the type (\ref{eq1}) from a Peierls-Nabarro model with
linear elasticity  see \cite{Gar_Muel,Con_Gar_Muel}. The analysis in the
three-dimensional case is substantially more subtle.
A formulation of dislocations in terms of currents was considered also in
\cite{ScalaVanGoethem2013}. 

The aim of this paper is to study the lower semicontinuity and relaxation of
functionals of the type (\ref{eq1}). One important question is whether
sequences of measures with the given properties and bounded energy converge,
upon taking a subsequence and in a suitable weak sense, to a measure in the
same class. Without the divergence-free constraint this is, in general, not
true. This can be solved by rephrasing the problem in terms of 1-rectifiable
currents. The same tool is also helpful for proving density results and a
structure theorem. However, the standard theory of currents deals with the
scalar case \cite{federer,GiaquintaModicaSoucek1998I}, whereas for
dislocations lattice-valued currents are needed. 
Some statements, such as
compactness, can be directly generalized from the scalar case working
componentwise, this is however not always the case, as for example in the
density result one must make sure that all components are approximated using
the same polyhedral (or piecewise affine) curve. Therefore we revisit in
Section \ref{seccurrents} some
of the classical proofs showing how they can be extended to the case of
interest here.

Very general results for group-valued currents are available, but not all
 cases which are relevant for us are covered. 
The theory of group-valued currents was firstly developed by Fleming
\cite{fleming}. He considers so-called polyhedral chains with
coefficients in a suitable abelian normed group $G$ and then works in its
closure, with respect to the flat norm. Essential results such as compactness
and approximability were proved by White in \cite{white1,white2}. 
 The approach we chose is quite different, relying on an explicit integral representation of group-valued $1$-currents, matching with \eqref{eq1} (see \cite{An_Ma} for a similar point of view). 
In Section \ref{seccurrents} we rephrase our problem in terms of 1-currents,
and we prove the polygonal approximation, density and structure theorems. 
In the rest of the paper, for notational simplicity, we use mostly the 
language of measures. 

The relaxation of the functional (\ref{eq1}) turns out to be an integral
functional of the same form but with a different integrand, see  Section 
\ref{secrelax}. 
As in the case of the relaxation of partition problems 
 \cite{AB1,AB2} the integrand  in the relaxed functional, that we call the
 $\calH^1$-elliptic envelope, is obtained by a cell
 formula, given in (\ref{relax}) below. 
In Lemma \ref{cellproblem} below we derive 
algebraic upper and lower bounds for the relaxation.
We remark that in general the two bounds do not coincide, as was proven in the
 two-dimensional case in \cite{AB2}, see also \cite{Caraballo2009}.
 For a specific  problem of physical interest, namely, dislocations in a cubic
 crystal, we  give in Section \ref{secexample}  an algebraic lower bound and
 an explicit expression for the  
 $\calH^1$-elliptic envelope in the case of  small Burgers vector.
An application of the tools derived here to the study of dislocations in a
three-dimensional discrete model of crystals, which has partly motivated the
present work,  will be discussed separately \cite{CoGaOrt}.

\section{Preliminary results on $\Z^m$-valued $1$-currents}
\label{seccurrents}
\subsection{Definitions and notation}
A $1$-current $T$ is a functional on the space of smooth compactly supported $1$-forms (vector fields in $\R^n$).
 We focus here on {\em rectifiable currents}, which  are still
 a satisfying generalization of curves (or surfaces, in dimension greater than
 $1$), but they are sufficiently regular to admit a handy representation as 
\begin{equation}\label{rect_curr_gr2}
\langle T,\varphi\rangle=\int_\gamma\theta(x)\langle\varphi(x);\tau(x)\rangle\,d{\cal H}^{1}(x)\,\in\R^m\,,\quad\forall\,\varphi\in C_c^\infty(\Omega,\R^n)\,,
\end{equation}
 where $\Omega\subseteq\R^n$ is open, $\gamma\subset\Omega$ is a
 $1$-rectifiable set  and $\tau:\gamma\to\Snmu$ is its 
 tangent vector, ${\cal H}^{1}$-almost everywhere. The multiplicity is an
 $L^1$ map 
\[
\theta:\gamma\to\Z^m\,.
\]
Let us point out that, setting $m=1$, we would recover the standard theory of
rectifiable currents  \cite{federer,Krantz_Parks, Mor1}; but,
for our aims, we need an actual lattice $\Z^m\subset\R^m$. Nevertheless, a
significant part of the theory of $\Z^m$-valued currents can be done
componentwise, reducing to the classical theory. 
Notice that the results stated and proved in this section for $\Z^m$-valued 
rectifiable $1$-currents can be actually given in the more general context of 
currents with multiplicity in a lattice $\mathcal{L}$, i.e., a discrete subgroup of 
$\R^m$ spanning the whole of $\R^m$. Since we never use the specific Euclidean 
norm of $\Z^m$, the two formulations are completely equivalent, for notational 
simplicity we focus on $\Z^m$.

 We will denote by $\calR_1(\Omega,\Z^m)$ the set of rectifiable $1$-currents and we will take \eqref{rect_curr_gr2} as a definition.
 Roughly speaking, one can imagine a rectifiable current as a countable sum of
 oriented simple Lipschitz curves with $\Z^m$-multiplicities (see Thm. 4.2.25
 in \cite{federer} and its corollaries) and we will establish this remark
 precisely in Theorem \ref{thm:structure}. If the map $\theta$ is piecewise
 constant on the support of $T$, say $\theta_{|\gamma_i}\equiv
 \theta_i\in\Z^m$ with $\supp T=\bigcup_i\gamma_i$
and $\gamma_i$ the image of a function
 $\tilde\gamma_i\in \Lip([0,1];\R^n)$, then for every $\varphi\in
 C^\infty_c(\Omega,\R^n)$ 
\begin{equation}\label{pwc_curr}
\langle T,\varphi\rangle=\sum_i
\theta_i\int_{\gamma_i}\langle\varphi;\tau\rangle\,d{\cal H}^1=
\sum_i\theta_i\int_0^1\varphi(\tilde\gamma_i(s))\tilde\gamma'_i(s)\,ds \,.
\end{equation}
The total variation of the rectifiable current in (\ref{rect_curr_gr2}) is the
measure $\|T\|=|\theta|\calH^1\LL \gamma$, its mass is 
\[
{\bf M}(T)=\|T\|(\Omega)=\int_\gamma |\theta|\,d{\cal H}^1,
\]
and it gives the ``weighted length'' of the current $T$ with respect to the
Euclidean norm $|\cdot|$ on $\Z^m$.  Indeed, in the piecewise constant
multiplicities case \eqref{pwc_curr} the mass of $T$ is the sum 
\begin{equation}\label{mass_easy}
{\bf M}(T)=\sum_i |\theta_i|{\cal H}^1(\gamma_i)\,.
\end{equation}
%as a consequence of the aforementioned Theorem \ref{thm:structure}.
Since we use the Euclidean norm on $\Z^m$, the mass of a vectorial current is
not, in general, the sum of the masses of the components. Using a different norm on $\Z^m$ would lead
to an equivalent norm on $\calR_1$.

Consistently with Stokes' Theorem, the boundary of a $1$-current $T$ is the $0$-current
\begin{equation*}
\langle\partial T,\psi\rangle=\langle T,{\rm d}\psi\rangle\quad\forall\,\psi\in C^\infty_c(\Omega)\,. 
\end{equation*}
A current $T$ is closed if $\partial T=0$. If $T$ is closed, then
\begin{equation}\label{eqdefpartialT}
\int_\gamma \theta(x)  D_\tau\psi(x) d\calH^1(x)=
0\quad\forall\,\psi\in 
W^{1,\infty}_0(\Omega)
\end{equation}
where $\gamma$, $\theta$ and $\tau$ are as in (\ref{rect_curr_gr2}) and
$D_\tau\psi(x)$ is the tangential derivative of $\psi$ at $x$ along $\gamma$.
The integral is well defined
since the Lipschitz function $\psi$ has a Lipschitz trace on the rectifiable 
set $\gamma$, and therefore a tangential derivative
$\calH^1$-almost everywhere on $\gamma$.
Formally, and in analogy to (\ref{rect_curr_gr2}), we can write (\ref{eqdefpartialT}) as
$\langle T,{\rm d}\psi\rangle=0$ (the two expressions are indeed identical if $\psi\in C^1_c$). 
To prove (\ref{eqdefpartialT}) let 
$\psi_\eps\in C^\infty_c(\Omega)$ be such that $\|D\psi_\eps\|_\infty\le 2\|D\psi\|_\infty$
and
$\|\psi-\psi_\eps\|_\infty\le\eps$. We claim that $D_\tau\psi_\eps(x)$ converges
weakly-$*$ in $L^\infty(\gamma,\calH^1)$ to $D_\tau\psi(x)$. Indeed, 
$D_\tau\psi_\eps(x)$ is uniformly bounded and therefore has a subsequence which
converges weakly-$*$ to some $g\in L^\infty(\gamma)$. For every $C^1$ curve
$\gamma_j$, the restriction to $\gamma_j$ of $\psi_\eps$ converges uniformly, and
hence weakly-* in $W^{1,\infty}(\gamma_j)$, to the restriction of $\psi$.
Therefore $g=D_\tau \psi$, $\calH^1$-almost everywhere on $\gamma$.
Using  $\theta\in L^1(\gamma,\calH^1)$,
$D_{\tau}\psi_\eps(x)=\langle{\rm d}\psi_\eps(x), \tau(x)\rangle$ 
and $\langle T,{\rm d}\psi_\eps\rangle=0$,
it follows that
\begin{equation*}
 \int_\gamma \theta(x)  D_\tau\psi(x) d\calH^1(x)
%=\lim_{\eps\to0}\int_\gamma \theta(x)  D\psi_\eps(x) \tau(x) d\calH^1(x)
=\lim_{\eps\to0}\int_\gamma \theta(x) \langle{\rm d}\psi_\eps(x), \tau(x)\rangle d\calH^1(x)=0\,.
\end{equation*}
 This proves (\ref{eqdefpartialT}).

If the multiplicity $\theta$ is piecewise constant as in (\ref{pwc_curr}), then
\[
\langle\partial
T,\psi\rangle=\sum_i\theta_i\left(\psi(\tilde\gamma_i(1))-\psi(\tilde\gamma_i(0))\right)\quad\forall\,\psi\in
C^\infty_c(\Omega)\,. 
\]

%We shall denote by $\calR_0$ the currents in $\calR$
%without boundary.
%$G$ is a discrete group, for example $G=\Z^m$.
We say that a rectifiable $1$-current is polyhedral  if its support $\gamma$ is the union of finitely many 
segments and $\theta$ is constant on each of them. We denote by $\calP_1(\Omega;\Z^m)$ the set of polyhedral $1$-currents. 

%We recall that an integral current is a
%rectifiable current $T$ with rectifiable boundary $\partial T$.

For a bi-Lipschitz map $f:\R^n\to\R^n$,
 $f_\sharp T$ is the current
\begin{equation}\label{eqactionfcurrent}
\langle f_\sharp T,\varphi\rangle = \int_{f(\gamma)} \theta(f^{-1}(y))
\langle
\varphi(y), \tau'(y)\rangle d\calH^1(y)\,,
\end{equation}
where $\tau'$ is the tangent to $f(\gamma)$ with the same orientation as $\tau$, $\tau'(f(x))=D_{\tau} f(x)/|\nabla_{\tau} f(x)|$. As above, 
$D_{\tau}f(x)$ denotes the tangential derivative of $f$ along $\gamma$, which exists $\calH^1$-almost everywhere on $\gamma$ since $f$ is Lipschitz
on $\gamma$;
if $f$ is differentiable in $x$ then $D_{\tau} f(x)=Df(x)\tau(x)$.
%Notice: this is the standard definition for currents, different from
%the one used for measures.

Alternatively, one can interpret rectifiable $1$-currents as measures. 
%of finite total variation $\calM(\Omega;\R^{m\times n})$ of the form
%$$
%\mu=\theta\otimes\tau \calH^1\LL \gamma,
%$$
%i%n the notation of (\ref{rect_curr_gr2}) (recall that here currents have finite mass, since $\theta\in L^1$).
%The condition of being without boundary translates in the condition that the measure is divergence-free, 
We say that a measure $\mu\in \calM(\Omega;\R^{m\times n})$ is divergence free if 
$$
\int_\Omega \sum_{j=1}^nD\varphi_j \,d\mu_{ij}=0\qquad \forall\varphi\in C^\infty_c(\Omega,\R^n), \,\,i=1, \dots, m
$$
which we shorten to $\partial\mu=0$.
We denote
by $\MBCC(\Omega)$ the set of  divergence-free measures 
$\mu\in\calM(\Omega;\R^{m\times n})$ of the form 
$$
\mu=\theta\otimes\tau \calH^1\LL \gamma,
$$
where $\gamma$ is a $1$-rectifiable set contained in $\Omega$, $\tau:\gamma\to 
\Snmu$ its tangent vector, and $\theta:\gamma\to \Z^m$ is $\calH^1$-integrable. 
Such a measure is divergence-free if and only if the corresponding current 
defined by (\ref{rect_curr_gr2}) is closed.
We identify closed currents in $\calR^1(\Omega;\Z^m)$  with measures in $\MBCC(\Omega)$. With this
identification the total variation of $\mu$ coincides with the mass of $T$,
$\mass(T)=|\mu|(\Omega)$. %, while the divergence of $\mu$ coincides with the boundary of $T$. 

\subsection{Density}\label{subsec:density}

Our first result is an extension of the density theorem, as given in the
scalar case for example in \cite[Theorem 4.2.20]{federer},
 to vector-valued currents. 
We formulate the density result on $\R^n$, a local version can be easily
deduced using the extension lemma discussed below. Although we find it more
natural to  phrase and prove the
theorem in terms of 1-currents, the entire argument can be easily formulated
in terms of measures supported on curves, with only notational changes.
\begin{theorem}[Density]\label{theo:density}
Fix $\e>0$ and consider a $\Z^m$-valued closed $1$-current
$T\in\calR_1(\R^n,\Z^m)$. Then there exist a bijective map
$f\in C^1(\R^n;\R^n)$, with inverse also $C^1$, and a closed polyhedral $1$-current
$P\in\calP_1(\R^n,\Z^m)$ such that 
\begin{equation*}
\mass(f_\sharp T-P)\le\e
\end{equation*}
and
\begin{equation*}
|\nabla f(x)-\Id|+|f(x)-x|\le\e\quad\forall\,x\in\R^n\ .
\end{equation*}
Moreover, $f(x)=x$ whenever $\dist(x,\supp T)\ge\e$. 
\end{theorem}
It is here important that a current $T$ without boundary  can be approximated
by polyhedral currents without boundary. The proof cannot be done 
componentwise, since this would increase the total mass by a factor (depending
on $m$), but follows closely the strategy used for currents with integer
mul\-ti\-pli\-ci\-ty \cite{federer}. For the sake of simplicity, we will prove the density result in the case of interest for this paper ($1$-dimensional currents without boundary), but the same proof can be performed for $\Z^m$-valued currents of generic dimension $k$.

\begin{proof}
By standard arguments on rectifiable sets, there is  a
countable family $\mathcal{F}$ of $C^1$ curves
such that $\|T\|(\Omega\setminus \cup\mathcal{F})=0$. We denote by
$\lambda$ a real parameter in the interval $(0,1)$, which will be chosen at
the end of the proof.

\begin{steplist}
  \item
We fix a point $x_0\in\gamma\in{\cal F}$
and assume   that, for some $\theta_0\in\Z^m\setminus\{0\}$,
\begin{equation}\label{eqdensts}
\lim_{r\to0}\frac{\|T-S\|(Q^{\tau}_r(x_0))}{r}=0\,,  
\end{equation}
where $S$ is the current defined by
$\langle S,\varphi\rangle = \int_{\gamma} \theta_0 \langle
\varphi(x),\tau(x)\rangle
d\calH^1(x)$, and $Q^{\tau}_r(x_0)$ is the cube of side $2r$, center in $x_0$
and one side parallel to the vector $\tau$, which is the tangent to $\gamma$
in $x_0$.

Without loss of generality we can assume $x_0=0$ and $\Tan_0\gamma=\R e_1$,
where ${e}_1$ is the first vector of the 
canonical  basis of $\R^n$. We denote by $Q_r$ the cube of center $0$, side $2r$ and sides parallel to the coordinate directions. 
Let $\e'>0$ be a small parameter chosen later.
For $r$ sufficiently small the set $\gamma\cap
Q_r$ is the graph of a $C^1$ function $g:(-r,r)\to\R^{n-1}$
with $g(0)=0$ and $\|g\|_{C^1}<\e'$. The function $\tilde g:(-r,r)\to\R^n$ 
defined as $\tilde g(x_1)=(0,g(x_1))$ obeys
\begin{equation*}%\label{rolle_style}
\|D\tilde g\|_{L^\infty((-r,r))}<\e' \hskip2mm\text{ and }\hskip2mm
\|\tilde g\|_{L^\infty((-r,r))}<\e' r\ .
\end{equation*}
We define the function $f\in C^1(\R^n;\R^n)$ as
$$
f(x)=x-\psi(x)\tilde g(x_1)\,,
$$
where $\psi\in C^\infty_c(Q_r;[0,1])$ obeys   $\psi\equiv 1$ on $Q_{\lambda
  r}$ and
\begin{equation*}%\label{est_der_bump}
\|\nabla\psi\|_{L^\infty}\le\frac{2}{(1-\lambda)r}\ .
\end{equation*}
For $2 \varepsilon'<1-\lambda$ the function $f$ is 
  bi-Lipschitz and maps $\gamma\cap Q_{\lambda r}$ into the segment
$(\R e_1)\cap Q_{\lambda r}$.
Moreover for sufficiently small $\e'$ (on a scale set by $\lambda$ and $\e$)
one has 
\begin{eqnarray}
|f(x)-x|+|\nabla f(x)-\Id|&\le&|\psi(x)\tilde g(x_1)|+|\psi(x)\nabla \tilde g(x_1)\otimes{e}_1|\nonumber\\
                          & + &|\tilde g(x_1)\otimes \nabla\psi(x)|\nonumber\\
													&
                                                                                                        <
                                                                                                        &\e'\left(r+1+\frac{2}{(1-\lambda)}\right)
                                                                                                        <\e \label{volevamo} 
\end{eqnarray}
and
\begin{equation}\label{lip_const_inv}
\|f^{-1}\|_{C^1}\le 1+\e\ .
\end{equation}
\begin{figure}[t]
\begin{center}
 \includegraphics{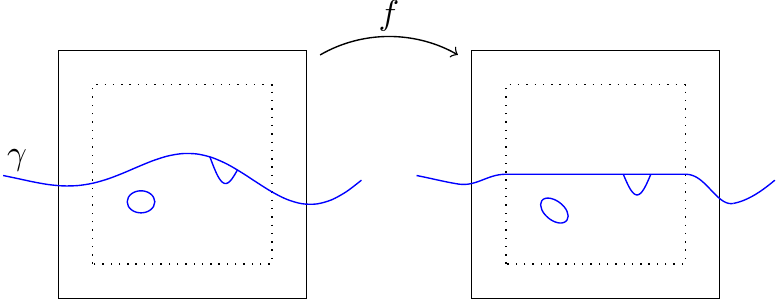}
\end{center}
\caption{The action of $f$ on $T$ in the proof of Theorem
  \ref{theo:density}. The inner cube is $Q_{\lambda r}$, the   outer one $Q_r$}
\label{fig:fs} 
\end{figure}
\item
We let $P$ be the polyhedral current defined by
\begin{equation*}
\langle P,\varphi\rangle = \theta_0 \int_{(-\lambda r,\lambda r)e_1}  \langle
\varphi,e_1\rangle\ d\calH^1\,.
\end{equation*}
With  $S$ as  in (\ref{eqdensts}), by definition of $P$ and $f$ we have \begin{equation*}
\mass(S\LL Q_r-f_\sharp^{-1}P)  =
|\theta_0|\, \calH^1\left(\gamma\cap (Q_r \setminus Q_{\lambda r})\right)\,.
\end{equation*}
Since $\gamma$ is a $C^1$ curve, 
\begin{alignat*}1
  \lim_{r\to0} \frac{\calH^1(\gamma\cap (Q_r \setminus Q_{\lambda r}))}{2r}
&=  (1-\lambda)\,.
\end{alignat*}
Using a triangle inequality and (\ref{volevamo}) we obtain
\begin{alignat*}1
\mass\left(f_\sharp\left(T\LL  Q_r\right)-P\right)&\le
\mass
\left(f_\sharp((T-S)\LL Q_{ r})\right)
+ \mass\left(f_\sharp(S\LL Q_{r}- f_\sharp^{-1}P) \right)\\
&\le(1+\eps)\mass
\left((T-S)\LL Q_{ r}\right)
+(1+\eps) \mass\left(S\LL Q_{r}- f_\sharp^{-1}P \right)
\end{alignat*}
and, recalling (\ref{eqdensts}), 
\begin{alignat*}1
\limsup_{r\to0}\frac{\mass\left(f_\sharp\left(T\LL  Q_r\right)-P\right)}{2r}
&\le(1+\eps)(1-\lambda)|\theta_0|\,.
\end{alignat*}
Since, again by (\ref{eqdensts}), $\|T\|\left(Q_r\right)/(2r)\to|\theta_0|$, 
 for $r$ sufficiently small
\begin{alignat}1\label{est_1point_compl}
\mass\left(f_\sharp\left(T\LL  Q_r\right)-P\right)
&<2(1-\lambda )\|T\|\left(Q_r\right)\ .
\end{alignat}

\item By \cite[Th. 4.3.17]{federer} for $\calH^1$-almost every point in the
union of the curves in $\mathcal{F}$ there is a $\theta_0$ with the property (\ref{eqdensts}), and 
therefore an $r_x\in(0,\eps/\sqrt{n})$ satisfying the property (\ref{est_1point_compl}) with $Q_r$ replaced by $Q_{r_x}^{\tau(x)}(x)$.
Using Morse's covering Theorem, we cover
$\|T\|$-almost all the set $\cup\mathcal{F}$ with a countable family of disjoint cubes
$Q_{r_k}^{\tau_k}(x_k)$ with $\tau_k=\tau(x_k)$ and sides $2r_k$, with $r_k<r_{x_k}$.
Then we have a  polyhedral $1$-current $P_k$ with support in
$Q_{r_k}^{\tau_k}(x_k)$ and a bi-Lipschitz map  $f_k\in C^1(\R^n;\R^n)$  satisfying
\eqref{volevamo}, \eqref{lip_const_inv} and \eqref{est_1point_compl}.

We choose a finite subfamily such that
\begin{equation}\label{est_covering_arg}
\sum_{k=1}^{K(\lambda)}\|T\|(Q_{r_k}^{\tau_k}(x_k))\ge \lambda \mass(T)
\end{equation}
and define
$$
f=f_1\circ\ldots\circ f_{K(\lambda )}\ .
$$ 
Since  $f_k(x)=x$ outside  $Q_{r_k}^{\tau_k}(x_k)$ for all $k$ and the cubes are
disjoint the condition  \eqref{volevamo} still holds
and $f(x)=x$ outside an $\eps$-neighbourhood of $\supp T$.

We define the polyhedral current
$$
P^I=\sum_{k=1}^{K(\lambda )} P_k\,,
$$
write
$$
f_\sharp T-P^I=\sum_{k=1}^{K(\lambda)} \left(f_\sharp\left(T\LL Q_{r_k}^{\tau_k}(x_k)\right)-P_k\right)+ f_\sharp \Big(T\LL\bigcup_{k>K(\lambda)}^\infty Q_{r_k}^{\tau_k}(x_k)\Big)
$$
and, recalling \eqref{est_1point_compl} and
\eqref{est_covering_arg},  conclude that 
\begin{equation}\label{est_kpoints_mass}
\mass(f_\sharp T-P^I)<2(1-\lambda)\mass(T)+(1-\lambda)\mass(T)=
3(1-\lambda)\mass(T)\ . 
\end{equation}

\item
We finally modify the polyhedral current $P^I$ to make it closed.

The current $f_\sharp T-P^I$ has multiplicity in $\Z^m$
and hence it can be decomposed
 in $m$ rectifiable scalar $1$-currents. Since $\partial f_\sharp T=0$ and $\partial
 P^I$ is a polyhedral current with finite mass (a finite sum of Diracs,
 actually) we can apply
 the Deformation Theorem in 
 \cite[Th. 4.2.9]{federer} to each component of $f_\sharp T-P^I$ in order
 to represent  it as
$$
f_\sharp T-P^I=P^O+Q+\partial S\ .
$$
Here $P^O,Q\in\calP_1(\R^n,\Z^m)$ are polyhedral $1$-currents satisfying 
\[
\mass(P^O)\le \sqrt{m}\, \cc_O(\mass(f_\sharp T-P^I)+\tilde\eps\mass(\partial P^I))
\] 
and
\[
\mass(Q)\le \tilde{\e}\sqrt{m}\, \cc_Q\mass\left(\partial P^I\right),
\]
for some $\tilde\varepsilon$ arbitrarily small, where
 $\cc_O,\cc_Q>0$ are geometric constants.
The current $Q$ is polyhedral by \cite[Th. 4.2.9(8)]{federer}.
since $\partial (f_\sharp T-P^I)$ is polyhedral. 

Then $P=P^I+P^O+Q$ is a closed polyhedral $1$-current with
\begin{eqnarray*}
\mass(f_\sharp T-P)&\le&\mass(f_\sharp T-P^I)+\mass(P-P^I)\\
                   &\le& 3(1-\lambda)\mass(T)+\mass(P^O+Q)\\
                   &\le& 3(1+\sqrt{m}\, \cc_O)(1-\lambda)\mass(T) + \tilde\e
                   \sqrt{m}\,(\cc_O+\cc_Q) \mass\left(\partial P^I\right).
\end{eqnarray*}
We first choose   a $\lambda\in (0,1)$ such that
the first term is less than $\frac12 \eps$, then $\tilde\e$ such that the
second term is also less than $\frac12\eps$, and conclude.
\end{steplist}
\end{proof}

 As a consequence of  Theorem~\ref{theo:density} we easily prove that any closed current  $T\in \calR_1(\R^n;\Z^m)$ can be approximated by sequences of polyhedral currents $P_k$ in the weak topology for currents, where
$$
 P_k\weakstarto T \quad \Longleftrightarrow\quad \langle P_k,\varphi\rangle\buildrel{k\to+\infty}\over{\longrightarrow} \langle T,\varphi\rangle\qquad \forall \ \varphi\in C^\infty_c(\R^n;\R^n)\,.
$$
We recall that  the currents $P_k$  are supported on a finite number of
segments. 
\begin{corollary}\label{cor:density}
  For every $T\in \calR_1(\R^n;\Z^m)$ with $\partial T=0$ there is a
  sequence of polyhedral currents $P_k\in \calP_1(\R^n;\Z^m)$ with $\partial P_k=0$ such
  that
\begin{equation*}
  P_k\weakstarto T \text{ and }\  \mass(P_k)\to \mass(T)\,.
\end{equation*}
\end{corollary}

We conclude this section  with an extension lemma, that can be found in
various forms in the literature. We sketch here the argument for the case of
interest, in which the closedness is preserved.
\begin{lemma}[Extension]\label{lemma:extension}
  Let $\Omega\subset\R^n$ be a bounded Lipschitz open set. For every closed rectifiable
  1-current defined in $\Omega$, $T\in \calR_1(\Omega;\Z^m)$, there is a closed rectifiable 1-current
  $\Ext T\in \calR_1(\R^n;\Z^m)$ with $\Ext T\LL\Omega=T$ and $\mass(\Ext T)\le \cc 
  \mass(T)$. The constant depends only on $\Omega$. Further,
$\lim_{\delta\to0} \mass(\Ext T\LL (\Omega_\delta\setminus\Omega))=0$, 
where $\Omega_\delta=\{x: \dist(x, \Omega)<\delta\}$.
\end{lemma}
\begin{proof}
{\it Step 1.} We first extend $T$ to a neighbourhood of $\Omega$.

  Choose a function $N\in C^1(\partial\Omega;\Snmu)$ such that 
$N(x)\cdot \nu(x) \ge \alpha>0$  for almost all $x\in\partial\Omega$, where
$\nu$ is the outer normal to $\partial\Omega$ and $\Snmu$ is the unit sphere in $\R^n$. For $\rho>0$ sufficiently small the function 
$g:\partial\Omega\times (-\rho,\rho)\to \R^n$, $g(x,t)=x+tN(x)$, is
 bi-Lipschitz. Let $D_\rho=g(\partial\Omega\times (-\rho,\rho))$ and 
$f:D_\rho\to D_\rho$ be defined by $f(g(x,t))=g(x,-t)$. Then $f$ is bi-Lipschitz
and coincides with its inverse.

We define $\tilde T=T - f_\sharp (T \LL (D_\rho\cap \Omega))$. 
Let $\varphi\in C^1_c(\Omega\cup D_\rho)$. Then, recalling (\ref{eqdefpartialT}) and
interpreting the duality in that sense,
\begin{equation*}
  \langle \tilde T, D\varphi\rangle = 
  \langle T, D\varphi\rangle -
  \langle f_\sharp (T \LL (D_\rho\cap \Omega)), D\varphi \rangle =
  \langle T, D\varphi-D((\varphi\chi_{D_\rho\setminus\Omega})\circ f)\rangle =0
\end{equation*}
since $\varphi-(\varphi\chi_{D_\rho\setminus\Omega})\circ f \in  W^{1,\infty}_0(\Omega)$, and $T$ is closed.

\sloppypar
{\it Step 2.} Let $\tilde \gamma$  and $\tilde \theta$ be the support and the
multiplicity of $\tilde T$, defined as in
(\ref{rect_curr_gr2}). We can slice the outer tubular neighborhood $D_\rho\setminus\Omega=g(\partial\Omega\times[0,\rho))$ 
through the family of sets $\partial(\Omega_s)$ with $s\in[0,\rho)$. More precisely, we slice
(see  \cite[Section 4.3]{federer} or \cite{Krantz_Parks}) the current $\tilde T\LL(D_\rho\setminus\Omega)$ with the 
distance function from $\partial\Omega$. By slicing, we get that
\[
\mass(\tilde T)\ge \int_0^\rho \Big( \sum_{x\in \tilde\gamma\cap \partial(\Omega_s)}   |\tilde\theta(x)| \Big) ds\,.
\]
Moreover, we can choose $s\in (0,\rho)$ such that 
\begin{equation*}
  \sum_{x\in \tilde\gamma \cap \partial(\Omega_s)}  |\tilde\theta(x)| \le \cc  \mass(T)\,,
  \end{equation*}
with a constant depending only on $\Omega$, and the sum runs over finitely many points $x_1, \dots, x_M$. Let us point out that the set of points $\{x_1,\dots, x_M\}$, with multiplicity $\tilde\theta(x_1), \dots,\tilde\theta(x_M)$ and positive orientation if $\tilde\gamma$ exits
$\Omega_s$ at $x_i$, are the boun\-da\-ry of $\tilde T\LL\Omega_s$.
For each $i=2,\dots, M$, let $
\gamma_i$ be a Lipschitz curve in $\R^n\setminus \Omega_s$ which joins $x_1$ with
$x_i$ and has length bounded by $C(\Omega)$. Let $\tau_i$ be the tangent vector,
with the same orientation as $\tilde\gamma$ in $x_i$. We
set 
\begin{equation*}
  \langle \Ext T,\varphi\rangle = \langle \tilde T\LL \Omega_s, \varphi\rangle + 
  \sum_{i=2}^M \tilde\theta(x_i) \int_{\gamma_i}  \langle D\varphi,\tau_i\rangle
  d\calH^1\,\quad\forall\,\varphi\in C^\infty_c(\R^n,\R^n). 
\end{equation*}
Since $T$ was closed one can see that $\Ext T$ is also closed.
To conclude the proof it is enough to note that, by construction,  $\mass
(\Ext T\LL\partial\Omega)=0$  and hence $\lim_{\delta\to0} \mass(\Ext T\LL
(\Omega_\delta\setminus\Omega))=0$. 
\end{proof}

\subsection{Compactness and Structure}\label{sec:compact}
In this section we characterize the support of rectifiable $1$-currents
without boundary as a countable union of loops. This characterization is
known in the theory of one dimensional integral currents (i.e. with scalar
multiplicity). In the latter case the result is stated in \cite{federer},
subsection 4.2.25, where a quick sketch of the proof is also given. Here,
for the convenience of the reader, we will give a complete proof. 

We start with the compactness statement, which is also used in proving the Structure Theorem \ref{thm:structure}.
\begin{theorem}[Compactness]\label{compactness}
  Let $(T_k)_{k\in\N}$ be a sequence of rectifiable $1$-currents 
  without boundary in $\calR_1(\R^n;\Z^m)$. If 
  \begin{equation*}
    \sup_{k\in\N} \mass(T_k) <\infty
  \end{equation*}
  then there are a current $T\in\calR_1(\R^n;\Z^m)$ without boundary and a subsequence 
 $(T_{k_j})_{j\in\N}$ such that
  \begin{equation*}
    T_{k_j} \weakstarto T  \,.
  \end{equation*}
\end{theorem}
\begin{proof}
  This follows from the result on scalar currents
\cite[Theorems 4.2.16]{federer}
 working componentwise.
\end{proof}

\begin{theorem}[Structure]\label{thm:structure}
  Let $T\in \calR_1(\R^n;\Z^m)$ with $\partial T=0$ and $\mass(T)<\infty$. Then there are countably many
  oriented Lipschitz closed curves $\gamma_i$ with tangent vector fields
  $\tau_i:\gamma_i\to   \Snmu$ and multiplicities $\theta_i\in \Z^m$ 
such that
  \begin{equation*}
    \langle T,\varphi\rangle =
    \sum_{i\in\N} \theta_i \int_{\gamma_i}\langle \varphi , \tau_i\rangle \, d\calH^1\,.
  \end{equation*}
Further,
\begin{equation*}
  \sum_i|\theta_i| \calH^1(\gamma_i)\le \sqrt{m} \,\mass(T)\,.
\end{equation*}
\end{theorem}
\begin{proof}
Since each current in $\calR_1(\Omega;\Z^m)$ can be  seen as the sum of $m$ rectifiable $1$-currents with scalar integer multiplicity, it suffices to prove the statement in the scalar case $m=1$.
From the density of polyhedral currents (see Corollary \ref{cor:density})
there is a sequence of polyhedral currents
without boundary $P_k\in  
\calR_1(\R^n;\Z)$ such that
\begin{equation*}
  P_k\weakstarto T \text{ and }\  \mass(P_k)\to \mass(T)\,.
\end{equation*}
Each $P_k$ can be decomposed into the sum of finitely many polyhedral loops,
\begin{equation*}
  P_k=\sum_{j=1}^{J_k} L_{j,k}\,,
\end{equation*}
such that 
\begin{equation}\label{mass-bound}
\sum_{j=1}^{J_k} \mass(L_{j,k})=  \mass(P_k)\leq M\,,
\end{equation}
for some $M>0$.

We can assume these loops $L_{j,k}$ to be ordered by mass, starting with the biggest one.
Moreover we can assume (up to extracting a subsequence) that the currents
$L_{j,k}$ have  multiplicity $1$
and that for every $j$ they weakly converge to some closed rectifiable $1$-current $L_j$. Let us denote by $\tilde T$ the current
$$
\tilde T=\sum_{j=1}^{\infty} L_{j}.
$$
We need to show that $\tilde T=T$. If $\mass(T)=0$ there is nothing to prove. 
Otherwise we fix  $\delta>0$ and observe that by
 \eqref{mass-bound}  we have $\mass(L_{i,k})<\delta$ for all $i>M/\delta$.
We write
\begin{equation}\label{decomposizione}
\langle P_k, \varphi\rangle = \sum_{i\leq \frac{M}{\delta}}\langle L_{i,k}, \varphi\rangle + \sum_{i> \frac{M}{\delta}}\langle L_{i,k}, \varphi\rangle\,.
\end{equation}
In the first sum of the right hand side we can take the limit as $k\to\infty$
and get $\sum_{i\leq \frac{M}{\delta}}\langle L_{i}, \varphi\rangle$.
Parametrizing each polyhedral curve by arc length, and possibly passing to a 
further subsequence, we see that each polyhedral curve converges to a closed Lipschitz curve.

The second sum in (\ref{decomposizione}) can be estimated as follows. For every $i>M/\delta$ and for every $k$ we fix a point $x^k_i\in\supp L_{i,k}=\gamma_{i,k}$ and using the fact that $\gamma_{i,k}$ is a closed curve we have
\begin{equation}\label{stima1}
\begin{split}
\Big|\sum_{i> \frac{M}{\delta}}\langle L_{i,k},
\varphi\rangle\Big|&=\Big|\sum_{i>
  \frac{M}{\delta}}\int_{\gamma_{i,k}}\langle
\varphi-\varphi(x^k_i), \tau^k_i\rangle\, d\calH^1\Big|\\& \leq  \sum_{i>
  \frac{M}{\delta}} \sup_{\gamma_{i,k}}|\varphi-\varphi(x^k_i)|
\mass(L_{i,k})\\ 
&\leq  \delta \|\varphi\|_{\Lip}\sum_{i> \frac{M}{\delta}} \mass(L_{i,k})\leq \delta M \|\varphi\|_{\Lip}\,.
\end{split}
\end{equation}
Then we get
$$
\Big|\langle T-\sum_{i\leq \frac{M}{\delta}}\langle L_{i},\varphi\rangle\Big| \leq o(1) +\Big|\langle P_k-\sum_{i\leq \frac{M}{\delta}}\langle L_{i,k},\varphi\rangle\Big|\leq o(1)+\delta M \|\varphi\|_{\Lip}
$$
which implies $T=\tilde T$ and hence
$$
T=\sum_{j=1}^\infty  \tau_j\calH^1\LL \gamma_j\,,
$$
with $\gamma_j=\supp L_j$ and $\tau_j$ the corresponding tangent vector. 
\end{proof}

\section{Relaxation}\label{secrelax}
\subsection{Main result}
In this section we consider the relaxation of functionals of the form
$$
E(\mu)=\begin{cases}\displaystyle \int_{\gamma}\psi(\theta, \tau)
\,d\calH^1& \text{ if } \mu=\theta\otimes \tau \calH^1\LL\gamma \in
\MBCCOmega,\\ +\infty & 
\hbox{ otherwise.}\end{cases} 
$$
We shall show that the relaxation is
$$
\bar E(\mu)=\begin{cases}\displaystyle \int_{\gamma}\bar\psi(\theta, \tau)
\,d\calH^1& \text{ if } \mu=\theta\otimes \tau \calH^1\LL\gamma\in
\MBCCOmega,\\ +\infty & 
\hbox{ otherwise,}\end{cases} 
$$
where $\bar\psi$ is defined by solving 
for any $b\in\Z^m$ and $t\in \Snmu$ a cell
problem, namely,
\begin{alignat}1
\bar\psi(b,t)=&\inf\Big\{ \displaystyle\int_\gamma
\psi(\theta,\tau)\, d\calH^1\,:\ \mu=\theta\otimes \tau \calH^1\LL\gamma\in
\MBCCOmega[B_{1/2}]\,,\nonumber \\  
&\hskip3.5cm \supp(\mu - b\otimes t
\calH^1\LL (\R t \cap B_{1/2})) \subset B_{1/2}\Big\}
\label{relax} 
\end{alignat}
where $B_{1/2}$ denotes a ball of radius $1/2$ and center $0$. The condition on
the support in (\ref{relax}) fixes the boundary values of $\mu$, in the sense that
it requires the existence of a ball $B'\subset\subset B_{1/2}$ 
with $\mu= b\otimes t
\calH^1\LL \R t$ on $B_{1/2}\setminus B'$.
We call the function $\bar\psi$ the {\it $\calH^1$-elliptic envelope of
  $\psi$} and say that $\psi$ is $\calH^1$-elliptic if $\bar\psi=\psi$. 
It is easy to see  that $\bar\psi(b,t)\leq\psi(b,t)$, and our result
implies that   $\bar\psi$ is the largest
$\calH^1$-elliptic  function below $\psi$.

For any open set $\omega\subset\Omega$, we write
\begin{equation*}
E(\mu,\omega)=\int_{\gamma\cap \omega}
  \psi(\theta,\tau) d\calH^1  
\end{equation*}
where $\mu=\theta\otimes \tau\calH^1\LL\gamma\in \MBCCOmega$, and the same for
$\bar E$. 
\begin{theorem}[Relaxation]\label{theo:relaxation}
Let $\psi:\Z^m\times \Snmu\to[0,\infty)$  be Borel measurable with
$\psi(b,t)\ge c_0|b|$ and $\psi(0, \cdot)=0$;
define  $\bar\psi$ as in  (\ref{relax}). Let $\Omega\subset\R^n$ be a
bounded Lipschitz set.
Then $\bar E$
is the lower semicontinuous envelope of
$E$
with respect to the weak convergence in $\MBCCOmega$, in the sense that
\begin{alignat*}1
\bar E(\mu) &= \inf\left\{\liminf_{j\to\infty}
E(\mu_j): \mu_j\in \MBCCOmega, \,\, \mu_j\weakstarto \mu \right\}\,.
\end{alignat*}
In particular, $\bar E$ is lower semicontinuous.
\end{theorem}

A key ingredient in the proof of the relaxation is to use the deformation theorem
to reduce to the case that the limit is
polyhedral. The continuity of   $\bar E$  under deformations 
follows from the Lipschitz continuity of
the integrand $\bar\psi$, see Lemma \ref{lemmaenergycontinuous}. 
In turn, the Lipschitz continuity of $\bar\psi$ is proven via a series of
constructions in Lemma \ref{cellproblem}.
The upper bound is then obtained covering the polyhedral with balls and using
the definition of $\bar\psi$. For the lower bound instead we need to show that 
$\bar E$ is lower semicontinuous on polyhedrals, which can  be done
locally assuming that the limit is a straight line. 
The most involved part of the proof deals with the relation between
minimization with boundary data and without boundary data, 
and  is discussed in Lemma \ref{lemma:boundary} below.

\subsection{Proof of the upper bound}
We start by proving the Lipschitz continuity of $\bar\psi$. As a side product
we also show that $\bar\psi$ (and hence any $\calH^1$-elliptic function), much
like the case of $BV$-elliptic integrands, is subadditive and convex.
\begin{lemma}[Cell problem]\label{cellproblem}
Let $\psi$, $\bar\psi$ be as in Theorem \ref{theo:relaxation}. Then:
\begin{enumerate}
  \item\label{polygonal}
    For every polyhedral measure $\mu=\sum_{i=1}^N b_i\otimes
  t_i\calH^1\LL\gamma_i\in \MBCCOmega[B_{1/2}]$ such that $\gamma_i\subset
  B_{1/2}$ are disjoint segments (up to the endpoints) and $\supp(\mu
  -b\otimes t\calH^1\LL (t\R\cap 
  B_{1/2}))\subset B_{1/2}$ one has
  \begin{equation*}
  \bar\psi(b,t)\leq \sum_{i=1}^N \calH^1(\gamma_i)\bar\psi(b_i,t_i)=\bar E(\mu, B_{1/2})    \,.
  \end{equation*}
  \item\label{convex} The function 
\begin{equation}\label{psiconvex}
t\mapsto \bar\psi\left(b,\frac{t}{|t|}\right)|t|
\end{equation}
is convex in $t\in\R^n$. In particular, $\bar\psi$ is continuous.
\item\label{subadditive} The function $\bar \psi$ is subadditive in its first argument, i.e.,
  \begin{equation*}
    \bar\psi(b+b',t)\leq\bar\psi(b, t)+ \bar\psi(b', t)\,.
  \end{equation*}
\item \label{sublinear} The function $\bar\psi$ obeys
  \begin{equation*}
    \frac1c|b|\le\bar\psi(b,t)\le c |b|
  \end{equation*}
  for all $b\in\Z^m$, $t\in \Snmu$.
\item\label{lipschitz} The function $\bar \psi$ is Lipschitz continuous in the sense that
  \begin{equation*}
    |\bar\psi(b,t)-\bar\psi(b', t')|\le c |b-b'|+ c |b|\, |t-t'|\,.
  \end{equation*}
\end{enumerate}
  with
  $\cc$ depending only on $\psi$.
\end{lemma}

\begin{figure}[t]
\begin{center}
 \includegraphics{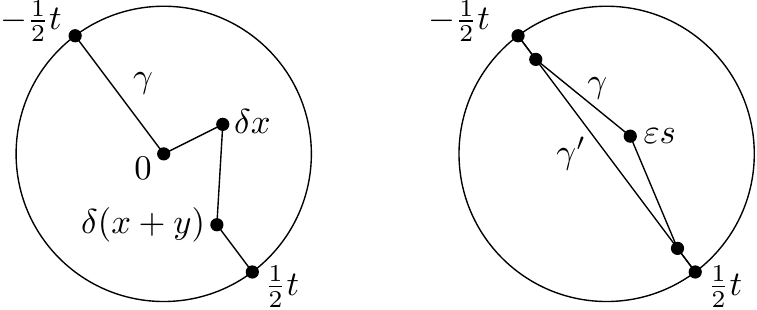}
\end{center}
\caption{Constructions used in the proof of Lemma \ref{cellproblem}\ref{convex}
  (left) and Lemma \ref{cellproblem}\ref{subadditive} (right).}
\label{fig:fig3} 
\end{figure}

\begin{proof}
\ref{polygonal}:
Let $B'$ be a ball such that 
$\supp(\mu
  -b\otimes t\calH^1\LL (t\R\cap 
  B_{1/2}))\subset B'\subset\subset B_{1/2}$.
We cover $\gamma=\cup_{i=1}^N\gamma_i\cap B'$
with a countable number of balls $\{B^k\}_{k\in\N}$ such that:
 the balls are disjoint and contained in $B'$;
 $\gamma\cap B^k$ is a diameter of $B^k$, $\mu\LL B^k
=b_{i_k} \otimes t_{i_k} \calH^1\LL (\gamma \cap B^k)$ for some $i_k\in \{1,
\dots, N\}$,
 $\calH^1(\gamma\setminus \cup_{k\in \N} B^k)=0$. 
Let $\eps>0$. By the definition of $\bar\psi$, for every $k$ we can find a
measure $\mu_k\in \MBCC(B^k)$ with $\supp(\mu_k - (\mu\LL B^k)) \subset B^k$ such that
\begin{equation*}
E(\mu_k, B^k)\leq \diam(B^k)\bar\psi(b_{i_k}, t_{i_k}) +\frac{\eps}{2^k}\,.
\end{equation*}
We define $\nu=\sum_k \mu_k + \mu\LL (B_{1/2}\setminus B')$. Then $\nu\in
\MBCCOmega[B_{1/2}]$ and 
$\supp(\nu -b\otimes t\calH^1\LL (t\R\cap
  B_{1/2}))\subset B_{1/2}$, therefore
\begin{alignat*}1
\bar\psi(b,t)&\leq E(\nu, B_{1/2})=\sum_{k\in\N} E(\mu_k, B^k)
+E(\mu, B_{1/2}\setminus B')\\
&\leq \sum_{i=1}^N
\calH^1(\gamma_i)\bar\psi(b_{i}, t_{i}) +
\psi(b,t) \calH^1(r\R\cap B_{1/2}-B')+
\eps\,. 
\end{alignat*}
We conclude by the arbitrariness of $B'$ and $\eps$. 

\noindent\ref{convex}:
We extend $\bar\psi$ to $\Z^m\times \R^n$ to be  one-homogeneous in the last
argument (i.e., to be the function given in  (\ref{psiconvex})). 
Let $\tilde x, \tilde y\in \R^n$, $\lambda\in (0,1)$. We want to show that
\begin{equation*}
  \bar\psi(b, \lambda \tilde x+(1-\lambda) \tilde y)\le \lambda  \bar\psi(b,
  \tilde x)+(1-\lambda) \bar\psi(b,\tilde y)\,. 
\end{equation*}
By the definition of the extension of $\bar\psi$, defining $x=\lambda \tilde
x$ and $y=(1-\lambda)\tilde y$ it suffices to show that
\begin{equation*}
  \bar\psi(b, x+ y)\le   \bar\psi(b,
  x)+ \bar\psi(b, y)\,. 
\end{equation*}
If $x+y=0$ then $\bar\psi(b,x+y)=0$ and the statement holds. If not, we choose $\delta>0$ such that
$\delta x, \delta x+\delta y\in B_{1/2}$ and define $t=(x+y)/|x+y|$.
Let $\gamma$ be the polyhedral curve
that joins (in this order) the points
\begin{equation*}
  -\frac12 t, \hskip2mm 0, \hskip2mm\delta x , \hskip2mm  \delta x +\delta y,
  \hskip2mm  \frac12 t\,,
\end{equation*}
see Figure \ref{fig:fig3}.
Notice that the first and last segment belong to the line $t\R$ and that
$\gamma\subset \overline B_{1/2}$.
We apply \ref{polygonal} to the measure $\mu=b\otimes \tau \calH^1\LL\gamma$,
where $\tau$ is the tangent to $\gamma$, and obtain
\begin{equation*}
  \bar\psi(b,t)\le (1- \delta |x+y|)\bar\psi(b,t)+ \delta |x| \bar\psi\left(b, \frac{x}{|x|}\right)
+ \delta |y| \bar\psi\left(b, \frac{y}{|y|}\right)\,.
\end{equation*}
Rearranging terms this gives $\bar\psi(b,x+y)\le \bar\psi(b,x)+\bar\psi(b,y)$,
as desired.

\noindent\ref{subadditive}:
Fix $\eps>0$ and a vector $s\in \Snmu$ not parallel to $t$. Let $\gamma$ be
the curve that joins the points 
\begin{equation*}
  -\frac12 t, \hskip2mm \left(-\frac12 +\eps\right) t, \hskip2mm \eps s, \hskip2mm   \left(\frac12
  -\eps\right) t,\hskip2mm \frac12 t\,,
\end{equation*}
see Figure \ref{fig:fig3}.
We define the polyhedral measure
\begin{equation*}
  \mu_\eps = b\otimes t \calH^1\LL (t\R\cap B_{1/2}) + b' \otimes \tau \calH^1\LL\gamma\,,
\end{equation*}
where $\tau$ is the tangent vector to $\gamma$. Notice that the supports of the
two components overlap on the two segments of length $\eps$ close to $\partial
B_{1/2}$.
By \ref{polygonal} we obtain
\begin{equation*}
\bar  \psi(b+b',t)\le (1-2\eps) \bar\psi(b,t) + 2\eps \bar\psi(b+b',t) + 
  \frac12 \bar\psi(b',t'_\eps) + \frac12 \bar\psi(b', t''_\eps).
\end{equation*}
Since $\bar\psi$ is continuous in the second argument, taking $\eps\to0$ proves
the assertion. 

\noindent\ref{sublinear}:
The lower bound is immediate from the definition of $\bar\psi$ and the growth of $\psi$. To prove the upper bound, we
deduce from \ref{subadditive} 
\begin{equation*}
  \bar\psi(b,t)\le \sum_{j=1}^n |b\cdot e_j| (\bar\psi(e_j,t)+\bar\psi(-e_j,t))
\end{equation*}
and observe that, since $\bar\psi$ is continuous,
\begin{equation*}
  \max_{j=1,\dots, n} \max_{t\in \Snmu} (\bar\psi(e_j, t)+\bar\psi(-e_j, t))<\infty\,.
\end{equation*}

\noindent\ref{lipschitz}:
From \ref{subadditive} and \ref{sublinear} we obtain
\begin{equation*}
  \bar\psi(b,t)\le \bar\psi(b',t)+c|b-b'|\,,
\end{equation*}
while by  \ref{convex} and \ref{sublinear} we deduce that
\begin{equation*}
  \bar\psi(b,t)\le \bar\psi(b, t')+ |t-t'|\bar\psi\left(b, \frac{t-t'}{|t-t'|}\right) 
  \le \bar\psi(b,t')+ c |b| \, |t-t'|\,.
\end{equation*}
\end{proof}

We now show that the continuity of $\bar\psi$  proven in \ref{lipschitz} gives
continuity of $E$ under deformations.
\begin{lemma}\label{lemmaenergycontinuous}
  Assume that $\psi:\Z^m\times \Snmu\to[0,\infty)$ is Borel measurable,
  obeys $\psi(0,t)=0$, $\psi(b,t)\ge c|b|$ and
  \begin{equation*}
    |\psi(b,t)-\psi(b', t')|\le c |b-b'|+ c |b|\, |t-t'|\,.
  \end{equation*}
  Let $\mu$, $\mu'\in \MBCCOmega$. Then for any open set
$\omega\subset\Omega$ we have
\begin{equation*}
  |E(\mu,\omega)-E(\mu',\omega)|\le \cc  |\mu-\mu'|(\omega)\,.
\end{equation*}
Further, if $f:\R^n\to\R^n$ is bi-Lipschitz then
for any open set $\omega\subset\R^n$ 
\begin{equation*}
  |E(\mu,\omega)-E(f_\sharp\mu,f(\omega))| \le \cc  E(\mu,\omega)
  \|Df-\Id\|_{L^\infty}\,.
\end{equation*}
\end{lemma}
We recall that 
in this paper  $f_\sharp$ denotes the action of $f$ on the current associated
to $\mu$, see (\ref{eqactionfcurrent}). In particular, if 
$\mu=\theta\otimes \tau\calH^1\LL \gamma$, then
\begin{equation*}
  f_\sharp\mu = \theta\circ f^{-1} \otimes \tilde\tau\calH^1\LL
  f(\gamma)\,,\hskip1cm \tilde \tau = \frac{ D_\tau f}{|D_\tau f|} \circ f^{-1}
\end{equation*}
where $D_\tau f$ denotes as in (\ref{eqactionfcurrent})  the tangential derivative.
\begin{proof}
  Let $\mu=\theta\otimes \tau \calH^1\LL \gamma$, 
$\mu'=\theta'\otimes \tau' \calH^1\LL \gamma'$.
To prove the first estimate we observe that $\tau=\pm \tau'$ $\calH^1$-a.e. on
$\gamma\cap\gamma'$. Changing the sign of $\theta'$ and $\tau'$ on the set where
$\tau=-\tau'$ we compute
\begin{equation*}
\int_{(\gamma\cup\gamma')\cap\omega} |\psi(\theta,\tau)-\psi(\theta',\tau')| d\calH^1 
\le c \int_{(\gamma\cup\gamma')\cap\omega} |\theta-\theta'| d\calH^1  \le c
|\mu-\mu'|(\omega), 
\end{equation*}
where we defined $\theta=0$, $\tau=\tau'$ on $\gamma'\setminus\gamma$ and
 $\theta'=0$, $\tau'=\tau$ on $\gamma\setminus\gamma'$.

To prove the second statement in the theorem we write, by the area formula,
\begin{equation*}
  E(f_\sharp\mu,f(\omega))=\int_{f(\gamma)\cap f(\omega)} \psi(\theta\circ
  f^{-1}, \tilde\tau) d\calH^1=\int_{\gamma\cap\omega} \psi(\theta,
  \tilde\tau\circ f  )  |Df\tau| d\calH^1
\end{equation*}
and observe that $|\tilde\tau\circ f-\tau|\le |Df-\Id|$.
\end{proof}

At this point we give the proof of the upper bound.
\begin{proof}[Proof of the upper bound in Theorem \ref{theo:relaxation}.]
We only need to deal with the case $\bar E(\mu,\Omega)<\infty$.
Let $\mu\in\MBCCOmega$. We need to construct a sequence of measures $\mu_k
 \in\MBCCOmega$ such that $\mu_k\weakstarto\mu$ and
 \begin{equation*}
   \limsup_{k\to\infty} E(\mu_k,\Omega)\le \bar E(\mu,\Omega)\,.
 \end{equation*}
By Lemma \ref{lemma:extension} we
can extend $\mu$ to a measure $\Ext\mu\in \MBCCOmega[\R^n]$, with 
\begin{equation*}
  \lim_{\delta\to0} |\Ext\mu|(\Omega_\delta\setminus\Omega)=0
\end{equation*}
(we recall that  $\Omega_\delta=\{x: \dist(x, \Omega)<\delta\}$).
By Theorem \ref{theo:density} there are
a sequence of polyhedral measures $\mu_k\in\MBCCOmega[\R^n]$ and a sequence of
$C^1$ and bi-Lipschitz maps 
$f_k$ such that
\begin{equation*}
|\mu_k-(f_k)_\sharp\Ext\mu|(\R^n)\to0\,,\hskip2mm
\|f_k-x\|_{L^\infty}\to0\,,\hskip2mm
\|Df_k-\Id\|_{L^\infty}\to0\,.
\end{equation*}
This implies
$\mu_k\weakstarto\Ext\mu$.
By Lemma \ref{lemmaenergycontinuous} and
Lemma~\ref{cellproblem}\ref{lipschitz} one obtains 
\begin{alignat*}1
 \bar E(\mu_k,\Omega)\le &
\bar E((f_k)_\sharp \Ext\mu, \Omega) + \cc \|\mu_k-(f_k)_\sharp\Ext\mu\|\\
\le &  \bar E(\Ext \mu,\Omega_{\delta_k})(1 + \cc \|Df_k-\Id\|_{L^\infty})
+ \cc \|\mu_k-(f_k)_\sharp\Ext\mu\|\,,
\end{alignat*}
where $\delta_k=\|f_k-x\|_{L^\infty}\to0$. Taking the limit we conclude
\begin{alignat*}1
\limsup_{k\to\infty}  \bar E(\mu_k,\Omega)\le &\bar E(\mu,\Omega)\,.
\end{alignat*}
Therefore it suffices to prove the upper bound for polyhedral measures
(since we are dealing with bounded subsets of $\MBCCOmega[\R^n]$, weak
convergence is metrizable).

Let  $\mu=\sum_{i=1}^N b_i\otimes
  t_i\calH^1\LL\gamma_i\in \MBCCOmega[\R^n]$ be a polyhedral measure, in the
  sense that the $\gamma_i$ are disjoint segments, $b_i\in\Z^m$,
  $t_i\in\Snmu$, for $i=1,\dots, N$. Let $\gamma=\cup_{i=1}^N \gamma_i$. 
  We choose  $\eps>0$ and cover $\gamma\cap\Omega$, up to an $\calH^1$-null set,
  with a countable number of disjoint balls $\{B^k=B_{r_k}(x_k)\}_{k\in\N}$ with $r_k<\eps$, which are contained in
$\Omega$ and have the property that 
$\gamma\cap B^k$ is a diameter of $B^k$ and $\mu\LL B^k
=b_{i_k} \otimes t_{i_k} \calH^1\LL (\gamma \cap B^k)$ for some $i_k\in \{1,
\dots, N\}$ (this is similar to the proof of Lemma
\ref{cellproblem}\ref{polygonal}, but here we take small balls to  ensure weak
convergence). 
By the definition of $\bar\psi$, for every $k$ we can find a
measure $\mu_k\in \MBCCOmega[B^k]$ with $\supp(\mu_k - b_{i_k}\otimes t_{i_k} 
\calH^1\LL (x_k+\R t_{i_k}\cap B^k)) \subset B^k$ such that
\begin{equation*}
E(\mu_k, B^k)\leq \diam(B^k)\bar\psi(b_{i_k}, t_{i_k}) +\frac{\eps}{2^k}\,.
\end{equation*}
Finally, define $\nu_\eps=\sum_k\mu_k$. We have
$$
E(\nu_\eps,\Omega) \leq \bar E(\mu,\Omega) +\eps\,
$$
and the desired recovery sequence is then obtained by letting $\eps\to 0$.
\end{proof}

\subsection{Proof of the lower bound}

In order
to prove the lower bound, we need to show that the boundary conditions in the definition of $\bar\psi$ can be substituted with an asymptotic condition. We start by working on a rectangle and showing that the energy
is concentrated on the line.
\begin{lemma}\label{lem:te}
Let $\psi$ and $E$ be as in Theorem \ref{theo:relaxation}. 
Given $b\in\Z^m$ and $t\in \Snmu$, we choose a rotation $Q_t\in\mathrm{SO}(n)$ with
$Q_te_1=t$ and for  $h,\ell>0$ we define  the parallelepiped
$R^t_{\ell,h}=Q_t\left[\left(-\frac{\ell}{2},\frac{\ell}{2}\right)
\times\left(-\frac{h}{2},\frac{h}{2}\right)^{n-1}\right]$ and the 
energy on the parallelepiped
\begin{multline}\label{trunc_ener}
\varphi(b,t,\ell,h)=\inf\Big\{\liminf_{k\to\infty} \frac1\ell
E(\mu_k,R^t_{\ell,h})\,:\mu_k\in\MBCCOmega[R^t_{\ell,h}],\\ 
\mu_k\stackrel{*}{\rightharpoonup}b\otimes t {\cal H}^1\LL(\R t\cap
R^t_{\ell,h})\Big\}\,.
\end{multline}
Then $\varphi$ does not depend on $\ell$ and $h$. We write
$\varphi(b,t,\ell,h)=\varphi(b,t)$. 
\end{lemma} 
\begin{proof}
The  statement is obtained through the following remarks.
We work here at fixed $b$ and $t$ and write for simplicity
$\phi(\ell,h)=\varphi(b,t,\ell,h)$. 
\begin{enumerate}
\item With a rescaling argument we get that
\begin{equation}\label{e_homog}
\phi(\ell,h)=\phi(\lambda \ell,\lambda h)\quad\forall\,\lambda>0\,.
\end{equation}
\item It is also immediate to notice that
\begin{equation}\label{e_incr}
\phi(\ell,h)\le \phi(\ell,H)\quad{\rm whenever}\ h\le H\,,
\end{equation}
by definition.
\item Moreover
\begin{equation}\label{e_sel}
\phi\left(\frac{\ell}{p},h\right)\le \phi(\ell,h)
\quad\forall\,p\in\N\setminus\{0\}
\end{equation}
by a selection argument. For example, if $p=2$, then \eqref{e_sel} is obtained
choosing for each $k$ the half of $R^t_{\ell,h}$ with energy less than
$\frac{1}{2} E(\mu_k,R^t_{\ell,h})$.  
\end{enumerate}
Thus our claim is proved, because by the previous three steps we have, for all
$h,\ell>0$ and all $p\in\N\setminus\{0\}$,
\begin{equation*}
\phi\left(\frac{\ell}{p},h\right)
\le\phi(\ell,h)
=\phi\left(\frac{\ell}{p},\frac{h}{p}\right)
\le \phi\left(\frac{\ell}{p},h\right)
\end{equation*}
hence equality holds throughout.
\end{proof}

The next lemma shows that $\varphi$, which was defined using weak convergence
instead of  boundary
values, is the same as  $\bar\psi$. This is the key step in which we show that
the natural upper and lower bounds coincide. 
\begin{lemma}%[Boundary values]
\label{lemma:boundary}
Let $\psi$, $\bar\psi$ and $\bar E$ be as in Theorem \ref{theo:relaxation},
 $\varphi$ as in Lemma \ref{lem:te}.
Then we have:
\begin{enumerate}
\item\label{cellproblemprojection} For every sequence
$ \mu_k\in\MBCCOmega[B_{1/2}]$ with $\mu_k \weakstarto \mu=b\otimes t
\calH^1\LL (\R t\cap B_{1/2})$ weakly in $\MBCCOmega[B_{1/2}]$ 
 one has
\begin{equation*}
\varphi(b,t)\le
 \liminf_{k\to\infty}  \bar E(\mu_k,  B_{1/2})\,.
\end{equation*}
\item\label{psieqvarphi} $\bar\psi(b,t)=\varphi(b,t)$.
\end{enumerate}
\end{lemma}
\begin{proof}
  \ref{cellproblemprojection}: 
We can assume the liminf to be a limit and to be finite. We first pass from
$\bar E$ to $E$ on the right-hand side.
By the upper bound proven in the previous section, for every $k$ there is a
sequence $\nu^{(k)}_h\weakstarto \mu_k$  in $\MBCCOmega[B_{1/2}]$ such that
\begin{equation*}
 \limsup_{h\to\infty} E(\nu^{(k)}_h, B_{1/2}) \le   \bar E(\mu_k, B_{1/2}) \,.
\end{equation*}
Since the weak convergence is metrizable on bounded sets we can take a
diagonal subsequence and obtain a sequence
$\tilde\mu_k$ which converges weakly to $\mu$ in $\MBCCOmega[B_{1/2}]$, with
\begin{equation*}
  \lim_{k\to\infty} E(\tilde\mu_k, B_{1/2})\le
  \lim_{k\to\infty}  \bar E(\mu_k, B_{1/2})\,.
\end{equation*}
Therefore it suffices to  show that $\varphi(b,t)\le   \liminf_{k\to\infty} E(\tilde\mu_k,
B_{1/2})$.

We fix $\ell\in (0,1)$ and then choose $h\ll 1$ such that
$R^t_{\ell,h}\subset B_{1/2}$. Then $E(\tilde\mu_k, R^t_{\ell,h})\le E(\tilde\mu_k,
B_{1/2})$ and, using Lemma \ref{lem:te},
\begin{equation*}
  \ell\varphi(b,t) \le \liminf_{k\to\infty}
E(\tilde\mu_k, R^t_{\ell,h})\le 
 \liminf_{k\to\infty} E(\tilde\mu_k, B_{1/2})\,.
\end{equation*}
Since $\ell\in(0,1)$ was arbitrary, the proof is concluded.

\noindent\ref{psieqvarphi}:
We choose $b\in\Z^m$, $t\in\Snmu$, and set $\mu=b\otimes t \calH^1\LL(\R t\cap
R^t_{1,1})$.
We start by defining  a version of
$\bar\psi$ where the ball is replaced by a cube,
\begin{alignat*}1
\tilde\psi(b,t)=&\inf\Big\{ E(\tilde\mu, R^t_{1,1})
:\ \tilde\mu\in \MBCC(R^t_{1,1})\,,\,\,
 \supp(\tilde\mu -\mu) \subset R^t_{1,1}\Big\}.
\end{alignat*}
It  suffices to show that $\varphi\le \bar\psi$, 
$\bar\psi\le \tilde\psi$ and 
$\tilde\psi\le\varphi$.

To prove  $\varphi\le \bar\psi$ let $\mu^*=\theta^*\otimes \tau^*\calH^1\LL
\gamma^*$ be one of the measures entering (\ref{relax}). We fill $R^t_{1,1}$ by
$2k+1$ scaled-down copies of $\mu^*$, for all $k\in\N$. Precisely, let $f^k_j(x)=\frac{1}{(2k+1)}
(x+jt)$ and set $\mu^k=\sum_{j=-k}^k (f^k_j)_\sharp \mu^*$. 
Since $Df^k_j=\frac1{2k+1}\Id$, 
for any test function $\varphi\in C^0_c(\R^n)$ we have
\begin{equation*}
\int \varphi d [(f^k_j)_\sharp
\mu^*] = \frac1{2k+1} \int (\varphi\circ f^k_j)  d\mu^*  
= \frac1{2k+1} \int \varphi\left(\frac{jt+x}{2k+1} \right)
d\mu^*(x)\,.
\end{equation*}
Then $\mu^k\in
\MBCCOmega[R_{1,1}]$, 
$\mu^k\weakstarto \mu$, and $E(\mu^k,R^t_{1,1}) = E(\mu^*, B_{1/2})$. Since $\mu^*$
was arbitrary, we obtain $\varphi\le \bar\psi$.

By covering most of the diameter of
$B_{1/2}$ with small squares one can easily
see that   $\bar\psi\le \tilde\psi$. 

We now show  $\tilde\psi\le\varphi $.
Choose a sequence  $\mu_k\stackrel{*}{\rightharpoonup}\mu$ in
$\MBCCOmega[R^t_{1,1}]$ such that
\begin{equation}\label{efi}
  \lim_{k\to\infty} E(\mu_k, R^t_{1,1})=\varphi(b,t)\,.
\end{equation}
By Lemma \ref{lem:te},
for any $h\in (0,1)$ we have
\begin{equation*}
\varphi(b,t)\le  \liminf_{k\to\infty} E(\mu_k, R^t_{1,h}) \,.
\end{equation*}
In particular,
\begin{equation}\label{eqnoenergyoutsideh}
  \limsup_{k\to\infty} E(\mu_k, R^t_{1,1}\setminus R^t_{1,h})=0\,.
\end{equation}
By the structure theorem (Th. \ref{thm:structure}) the measure $\mu_k$ has the form
$\sum_i \theta_{k,i}\otimes \tau_{k,i}\calH^1\LL\gamma_{k,i}$, with 
$\theta_{k,i}\in\Z^m $ and $\gamma_{k,i}$ Lipschitz curves, each either closed 
or with endpoints in $\partial R^t_{1,1}$.
We denote by $J_k$ the set of $i$ for which the curve $\gamma_{k,i}$ intersects
 $R^t_{1,h}$, and we define
 $\gamma_k^{\circ}=\cup_{i\in J_k} \gamma_{k,i}$ and
 $\mu_k^\circ=\sum_{i\in J_k}\theta_{k,i}\otimes \tau_{k,i}\calH^1\LL\gamma_{k,i}$.
 By construction $\partial \mu_k^\circ=0$.
By (\ref{eqnoenergyoutsideh}) we have
\[
{\cal H}^1\left(\gamma_k\cap R^t_{1,1}\setminus R^t_{1,h}\right)\longrightarrow
0\quad{\rm as}\ k\to\infty\,, 
\]
therefore $\gamma_k^{\circ}\subset R^t_{1,2h}$ for $k$ sufficiently large.
In summary, we have constructed a new sequence of vector-valued measures $\mu_k^\circ$ which satisfies
\[
\mu_k^{\circ}\stackrel{*}{\rightharpoonup}\mu
\]
with ${\rm supp}\mu_k^{\circ}\subset R^t_{1,2h}$ and $\partial \mu_k^{\circ}
=0$ in $R^t_{1,1}$ (see Figure \ref{fig:cp2}). 

 As a consequence of the definition of the truncated energy in Lemma
 \ref{lem:te} we get 
\[
 (1-2h)\varphi\le
\liminf_{k\to\infty} E(\mu_k^{\circ},R^t_{1-2h,2h})\,,
\]
thus  the endstripes  $S^t_h=R^t_{1,2h}\setminus R^t_{1-2h,2h}$ contain little
energy, in the sense that
\begin{equation}\label{limsupfi}
\limsup_{k\to\infty} E(\mu_k^{\circ}, S^t_h)\le 2h\varphi\,.
\end{equation}

\begin{figure}[t]
\begin{center}
 \includegraphics[width=\linewidth]{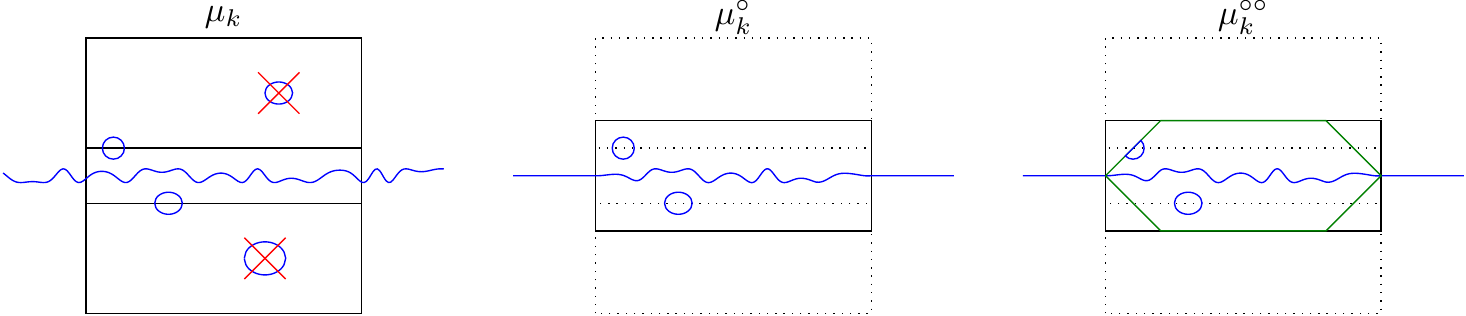} 
\end{center}
\caption{Passing from $\mu_k$ to $\mu_k^{\circ\circ}$.
The squares represent $R_{1,1}^t$, the rectangles $R_{1,h}^t$ and $R_{1,2h}^t$.
}
\label{fig:cp2} 
\end{figure}

As we drew in Figure \ref{fig:cp2}, we head to the conclusion squeezing the
measure $\mu_k^{\circ}$  through the projection $f^t:R^t_{1,2h}\to
R^t_{1,2h}$, defined by  $f^t(x)=x$ for $x\in R^t_{1-2h,2h}$ and $f^t(x)=Q_t f(Q_t^{-1}x)$ 
in $S^t_h$, where $Q_t$ is a rotation such that $Q_te_1=t$ and $f$ is defined as
\[
f(x_1,x')=\left(x_1,\left(\frac{1}{2h}-\frac{1}{h}|x_1|\right)x'\right)
\text{ for }x=(x_1,x')\in S^{e_1}_{h}\,.
\]
Let us define
\[
\mu_k^{\circ\circ}=f_\sharp^t (\mu_k^{\circ}).
\]
Thus
\[
E(\mu_k^{\circ\circ},S^t_h)\le \cc E(\mu_k^{\circ},S^t_h),
\]
and therefore by \eqref{efi} and \eqref{limsupfi}
\begin{equation}\label{quasiconcl}
  \limsup_{k\to\infty} E(\mu_k^{\circ\circ},R^t_{1,2h})\le \varphi+ ch \varphi\,.
\end{equation}

Finally we deal with the
boundary. By the definition of  $\mu_k^{\circ\circ}$,
\begin{equation}\label{duedelta}
\partial \mu_k^{\circ\circ}=\theta'\left(\delta_{1/2 e_1}-\delta_{-1/2 e_1}\right)\,.
\end{equation}
The measure
\[
\mu_k^{\circ\circ\circ}=\mu_k^{\circ\circ}+\theta'\otimes t {\cal H}^1\LL(\R e_1\setminus R^t_{1,h})
\] 
satisfies $\partial \mu_k^{\circ\circ\circ}=0$, but, at the same time,
\[
\mu_k^{\circ\circ\circ}\stackrel{*}{\rightharpoonup} b\otimes t {\cal
  H}^1\LL(\R e_1\cap R^t_{1,h})+\theta'\otimes t {\cal H}^1\LL(\R e_1\setminus
R^t_{1,h})\,, 
\]
thus $\theta'=b$. Thus, \eqref{duedelta} together with \eqref{quasiconcl}
implies $\tilde\psi\leq \varphi$.
\end{proof}

We are now ready for proving the lower bound.
\begin{proof}[Proof of the lower bound in Theorem \ref{theo:relaxation}.]
Fix  $\mu\in \MBCCOmega$ and consider a sequence $\mu_k\weakstarto \mu$.
Since $\bar E\le E$, it suffices to prove that
\begin{equation*}
  \bar E(\mu,\Omega)\le \liminf_{k\to\infty} \bar E(\mu_k,\Omega)
\end{equation*}
 (this means, it suffices
 to show that $\bar E$ is lower semicontinuous).
Passing to a subsequence we can assume that the sequence $\bar E(\mu_k,\Omega)$ converges. 
We can assume that  the limit is finite, and therefore that
$\sup_k|\mu_k|(\Omega)<\infty$. 
We extend each of the measures $\mu_k$ to $\Ext\mu_k\in \MBCCOmega[\R^n]$ using
Lemma \ref{lemma:extension}.
The sequence $\Ext\mu_k$  is uniformly bounded, extracting a subsequence we can
assume that $\Ext\mu_k$ has a weak limit, which is automatically an extension of
$\mu$.  With a slight abuse of notation we denote the limit by $\Ext\mu$.
We identify $\Ext\mu$ and
$\Ext\mu_k$ with the corresponding closed currents $T, T_k\in \calR_1(\R^n;\Z^m)$.

We fix $\eps>0$ and apply the Deformation Theorem to $\Ext\mu$
(Theorem~\ref{theo:density}). Let $f$ and $P$ be the 
resulting $C^1$ bi-Lipschitz map and polyhedral measure  such that
\begin{equation*}
\|f_\sharp \Ext\mu-P\|<\eps \text{ and } |f(x)-x|+|\nabla f(x)-\Id|<\eps\,.
\end{equation*}
We define   
\begin{equation*}
\tilde \mu_k = f_\sharp (\Ext\mu_k - \Ext\mu)+P=f_\sharp \Ext\mu_k - (f_\sharp
\Ext\mu-P)\,. 
\end{equation*}
Clearly $\partial 
\tilde \mu_k=0$; from $\Ext\mu_k\weakstarto \Ext\mu$ we deduce $\tilde \mu_k
\weakstarto 
P$. From Lemma \ref{lemmaenergycontinuous} 
we get, for $\omega_\eps=\{x\in\Omega: \dist(x,\partial\Omega)>\eps\}$,
\begin{equation}\label{approxlower}
\bar E(\tilde \mu_k,\omega_\eps) \le
 (1+\cc \|Df-\Id\|_{L^\infty}) \bar E(\mu_k,\Omega) +\cc  \|f_\sharp\Ext\mu-P\|\,.
\end{equation}
Since $P$ is  polyhedral, we can find finitely many disjoint balls
$B_i=B(x_i,r_i)\subset\omega_\eps$ such that $P\LL B_i=b_i\otimes t_i
\calH^1\LL (x_i+t_i \R\cap B_i)$ and $|P|(\omega_{2\eps}\setminus \cup B_i)\le \eps$. 
For each ball, by 
Lemma \ref{lemma:boundary}, we have
\begin{equation*}
  \bar E(P, B_i)=2r_i \bar\psi(b_i, t_i)\le \liminf_{k\to\infty} 
  \bar E(\tilde \mu_k, B_i)\,.
\end{equation*}
Summing over the balls we conclude that
\begin{equation*}
  \bar E(P, \omega_{2\eps}) \le \sum_i   \bar E(P, B_i) + \cc  
|P|(\omega_{2\eps}\setminus \cup B_i) 
\le \liminf_{k\to\infty} \bar E(\tilde \mu_k, \omega_\eps) + \cc   \eps.
\end{equation*}
By \eqref{approxlower} we then get
\begin{equation*}
  \bar E(P, \omega_{2\eps})\le (1+c\eps) \liminf_{k\to\infty} \bar
  E(\mu_k,\Omega)+\cc\eps. 
\end{equation*}
Since another application of Lemma \ref{lemmaenergycontinuous} gives
\begin{equation*}
  \bar E(\mu,\omega_{3\eps}) \le \bar E(P,\omega_{2\eps})(1 + \cc \eps) + \cc \eps\,,
\end{equation*}
the conclusion follows 
by the arbitrariness of $\eps$. 
\end{proof}

\section{Explicit relaxation for dislocations in cubic crystals}
\label{secexample}
We consider here the energy density $\psi: \Z^n\times \Snmu\to\R$
\begin{equation}\label{eqdefpsi}
  \psi(b,t)=|b|^2+\eta(b\cdot t)^2  
\end{equation}
which arises in the modeling of dislocations in crystals. Focusing on the case
$\eta\in[0,1]$ which arose in previous works 
\cite{Gar_Muel,CacaceGarroni2009,Con_Gar_Muel}, we  determine here the 
 relaxation $\bar\psi(b,t)$ for the (most relevant) small values of $b$
and in particular show that complex res may arise, in which different
values of $b$ and of $t$ interact. 

\subsection{Line-energy of dislocations}
A dislocation is a line singularity in a crystal which may be described by a
divergence-free measure of the form $\theta\otimes \tau\calH^1\LL\gamma$, as
studied in the previous sections, where $\theta$ physically represents the
components of the Burgers vector in a lattice basis \cite{HirthLothe1968,HullBacon}.
In the case that dislocations are restricted to a plane,
$\gamma\subset \R^2\times \{0\}$ and $\theta\in \Z^2$, 
a model of this form was derived from linear three-dimensional elasticity in
\cite{Gar_Muel,Con_Gar_Muel} using the tools of $\Gamma$-convergence, building
mathematically upon the concept of $BV$-elliptic envelope and physically 
upon a generalization of the Peierls-Nabarro model introduced by 
Koslowski, Cuiti{\~n}o and Ortiz \cite{Kos_Cui_Or,KoslowskiOrtiz2004}.
One key observation was that the (unrelaxed) energy per unit length of a
dislocation is given by a specific function
$\psi^c(b,t)$, which can be computed from the elastic constants of the
solid. 
Assuming a cubic kinematics for the dislocations and isotropic elastic constants
and writing  $t=(\cos\alpha,\sin\alpha)\in
\Suno$, the energy density takes the form
(see \cite[Eq. (51)]{CacaceGarroni2009} or \cite[Eq. (1.8)]{Con_Gar_Muel}), 
\[
\psi^c(b,t)=\frac{\mu a_0^2}{4\pi(1-\nu)}\,b\left(
\begin{array}{cc}
2-2\nu\cos^2\alpha & -2\nu\sin\alpha\cos\alpha\\
-2\nu\sin\alpha\cos\alpha   & 2-2\nu\sin^2\alpha
\end{array}
\right)b\,,
\]
where the parameter $\nu\in[-1,1/2]$ represents the material's Poisson ratio,
$\mu$ the shear modulus of the crystal, $a_0$ the length of the Burgers vector
(i.e., the lattice spacing).
Straightforward manipulations permit to rewrite this expression as
\begin{eqnarray}
\psi^c(b,t)
&=&\frac{\mu a_0^2}{4\pi(1-\nu)}\left(2(1-\nu)|b|^2+2\nu(b^\perp\cdot
  t)^2\right)=\frac{\mu a_0^2}{2\pi}\psi(b^\perp,t)\,,\label{c20} 
\end{eqnarray}
where $\psi$ was defined in (\ref{eqdefpsi}), 
$\eta=\frac{\nu}{1-\nu}\le 1$, and $b^\perp=(-b_2,b_1)$.
Without loss of generality we can assume $\eta\in[0,1]$: indeed, if $\nu<0$, we can rewrite \eqref{c20} as
$\psi^c(b,t)=\frac{\mu a_0^2}{2\pi(1-\nu)}\psi'(b,t)$
where $\psi'(b,t)=|b|^2+\eta'(b\cdot t)^2$ contains the 
constant $\eta'=-\nu\in[0,1]$. 

The expression  (\ref{eqdefpsi}) is invariant under rotations, and indeed 
the above discussion can be immediately generalized to the three-dimensional
case, resulting (at least in the somewhat academic case $\nu<0$) in the same
formula, see, e.g., \cite[Sect. 4.4]{HullBacon} or
\cite[Eq. (51)]{Kos_Cui_Or}. 

\subsection{Lower bound on the relaxation}
We now start the analysis 
of the energy density (\ref{eqdefpsi}).  The key idea is to decompose the
set  $\gamma$ on which the measure is concentrated into sets on which $\theta$ is
constant. Each component is then replaced by a segment with the same 
end-to-end span, an operation which by convexity does not increase the energy
(here we use Lemma \ref{lemmatildepsisubadditive} below). 
This involves an implicit rearrangement, which one can expect 
to be sharp since $\gamma$ is one dimensional. In a second step we show that only
small multiplicities are relevant in the definition of the relaxation, due to the
quadratic growth of $\psi$ (here we use Lemma \ref{lemmapsiinequalities} below).
A similar procedure is also helpful to characterize the relaxation in a
total-variation model for the reconstruction of optical flow in image
processing \cite{CoGiRu}.
\begin{proposition}\label{nnove}
Let $\eta\in [0,1]$, $\psi$ be as in (\ref{eqdefpsi}).
For $n\le 9$ its $\calH^1$-elliptic envelope obeys
\begin{equation}\label{eqlbtalpha}
\bar\psi(b,t)\ge \min\left\{\!\sum_{\alpha\in \{-1,0,1\}^n} \!\!\!\!\psie(\alpha, T_\alpha):
T\in \R^{n3^n}\,,
\!\!\!\!\sum_{\alpha\in \{-1,0,1\}^n} \!\!\!\!\alpha\otimes T_\alpha = b\otimes t\right\},
\end{equation}
where $\psie$ denotes the positively one-homogeneous extension of $\psi$,
\begin{equation}\label{eqdeftildepsi}
  \psie(b,t)=|t|\psi\left(b,\frac{t}{|t|}\right)\,.
\end{equation}
For $n\ge 10$ equation (\ref{eqlbtalpha}) holds with  $T\in\R^{n(4n+1)^n}$ and
 both sums  running over all $\alpha$ in $[-2n,2n]^n\cap \Z^n$.
\end{proposition}
\begin{proof}
  \begin{steplist}
\item We fix $b$ and $t$.
Let  $\mu=\theta\otimes\tau\calH^1\LL\gamma$ be any of the measures entering (\ref{relax}).
We decompose its support $\gamma$   depending on the value of $\theta$. 
For
any $\alpha\in \Z^n$ we set
\begin{equation*}
\gamma_\alpha=\{x\in\gamma:\,\theta(x)=\alpha\}\,.
\end{equation*} 
These countably many 1-rectifiable sets are pairwise disjoint and cover
$\gamma$. 
Since $\partial(\mu-b\otimes t\calH^1\LL(\R t\cap B_{\frac12}))=0$
we have 
\begin{equation*}
  b\otimes t = \int_\gamma \theta\otimes \tau d\calH^1 = \sum_{\alpha\in\Z^n} \alpha
  \otimes T_\alpha\,,
\end{equation*}
where we defined
\begin{equation*}
T_\alpha=\int_{\gamma_\alpha}\tau\,d\calH^1\,.
\end{equation*} 
An analogous decomposition of the energy gives
\begin{equation*}
E(\theta\otimes\tau\calH^1\LL\gamma)=\sum_\alpha\int_{\gamma_\alpha} \psi(\alpha,\tau)\,d\calH^1\ge \sum_\alpha \psie(\alpha, T_\alpha),
\end{equation*}
where in the second step we used
Lemma \ref{lemmatildepsisubadditive} below.
In particular, if the energy is finite then $\sum_\alpha |T_\alpha|<\infty$.

\item\label{stepthetasmall}  
Assume first $n\le 9$. Let $T:\Z^n\to \R^n$ be as above, 
 $\alpha^*\in\Z^n$ be such that $|\alpha^*_i|>1$ for some $i$ and
 $T_{\alpha^*}\ne0$. 
Let $a\in\Z^n$ be as in Lemma
 \ref{lemmapsiinequalities}\ref{lemmapsiinequalities1} below, so that
 \begin{equation*}
   \psie(\alpha^*-a, T_{\alpha^*})+   \psie(a, T_{\alpha^*})\le
   \psie(\alpha^*, T_{\alpha^*}) \,. 
 \end{equation*}
By the subadditivity in Lemma \ref{lemmatildepsisubadditive},
\begin{equation*}
  \psie(a,T_a+T_{\alpha^*})\le   \psie(a,T_{\alpha^*})+  \psie(a,T_{a})
\end{equation*}
and the same for $\alpha^*-a$.
We set $T'_{\alpha^*}=0$, $T'_a=T_a+T_{\alpha^*}$, $T'_{\alpha^*
-a} = T_{\alpha^*-a}+T_{\alpha^*}$, $T'_\alpha=T_\alpha$ for the other
values. Then $\sum_\alpha\alpha\otimes T'_\alpha=\sum_\alpha \alpha\otimes T_\alpha$ and
\begin{equation*}
  \sum_\alpha \psie(\alpha, T_\alpha')\le
  \sum_\alpha \psie(\alpha, T_\alpha)\,.
\end{equation*}
Let $M>2$.
Finitely many iterations of this step produce a $T^M$ with $T^M_\alpha=0$ for
all 
$\alpha$ with $\max_i|\alpha_i|\in [2,M]$. Taking the limit $M\to\infty$ 
gives a $T^\infty$ with $T^\infty_\alpha=0$ whenever 
$\max_i|\alpha_i|\ge 2$. This concludes the proof for $n\le 9$.

If $n\ge 9$ we use the same procedure with Lemma
\ref{lemmapsiinequalities}\ref{lemmapsiinequalities1largen} instead of
\ref{lemmapsiinequalities1}. 
  \end{steplist}  
\end{proof}

One key ingredient in the above proof was the subadditivity of $\psie$.
\begin{lemma}\label{lemmatildepsisubadditive}
The function $\psie$ defined in 
 (\ref{eqdeftildepsi})
is subadditive in the second argument, in the sense
that for any $b\in \Z^n$ and any set of vectors $T_1,\dots, T_N\in \R^n$ we have
\begin{equation*}
   \psie(b,\sum_i T_i)\le \sum_i \psie(b,T_i)\,.
\end{equation*}
Analogously, if $\gamma$ is $1$-rectifiable and $\tau$ its tangent,
\begin{equation*}
\psie\left(b,  \int_\gamma \tau d\calH^1 \right)\le
\int_\gamma\psie(b,\tau)d\calH^1 \,.
\end{equation*}
\end{lemma}
\begin{proof} 
For brevity we prove only the first formula, the differences are purely
notational. We can assume $b\ne 0$. We set $\tau_i=T_i/|T_i|$,  $L=\sum_i |T_i|$, and write
$ \psie(b,T_i)= |T_i| \varphi(\tau_i)$ 
where $\varphi(\tau)=|b|^2+\eta (b\cdot \tau)^2$, $\tau\in \R^n$. Since $\varphi$ is convex we
obtain 
\begin{equation*}
|b|^2 + \eta (b\cdot \hat \tau)^2=   \varphi(\hat\tau)\le
 \sum_i \frac{|T_i|}{L} \varphi(\tau_i)=
\frac{1}{L} \sum_i \psie(b,T_i),
\end{equation*}
where
\begin{equation*}
    \hat \tau = \sum_i  \frac{|T_i|}{L} \tau_i = 
\frac{1}{L} \sum_i T_i\,.
\end{equation*}
Set now 
 $h(\ell)=\ell|b|^2+\ell^{-1}\eta (b\cdot\hat\tau)^2$. 
 The function $h$  has a global minimum at
 $\ell_0=\sqrt{\eta}\frac{|b\cdot\hat\tau|}{|b|}\le|\hat\tau|$ and is
 increasing afterwards. 
Since $\hat\tau$ is an average of unit vectors, $|\hat\tau|\le 1$. We obtain
\begin{equation*}
  \psie(b,\hat\tau)=h(|\hat\tau|) \le h(1)=\varphi(\hat\tau),
\end{equation*}
and therefore the desired inequality
\begin{equation*}
  \psie(b,\sum_i T_i)=
  \psie(b,L\hat\tau)=
  L\psie(b,\hat\tau)\le L\varphi(\hat\tau)\le \sum_i \psie(b,T_i)\,.
\end{equation*}
\end{proof}

\begin{lemma}\label{lemmapsiinequalities}

  \begin{enumerate}
  \item\label{lemmapsiinequalities1}
    Let $n\in\{2,\dots, 9\}$, $b\in \Z^n$. If $\beta=\max_i |b_i|>1$ then
    there is  a vector
$a\in \Z^n$ such that $\max_i |a_i|=1$, $\max_i |b_i-a_i|=\beta-1$, and
\begin{equation}\label{eqpsisubadd}
\psi(b-a,t)+  \psi(a,t)\le \psi(b,t) \text{ for all } t\in \Snmu\,.
\end{equation}
  \item\label{lemmapsiinequalities1largen}
    Let $b\in \Z^n$. If $|b|\ge4\sqrt n$ then
    there is  a vector
$a\in \Z^n$ such that $\max_i |a_i|<\max_i|b_i|$ , $\max_i |b_i-a_i| <\max_i|b_i|$,
and (\ref{eqpsisubadd}) holds.
  \item\label{lemmapsiinequalities2}
  If $a,b\in\Z^n$, $a\cdot b=0$  and $|b|\le |a|\sqrt2$, then
\begin{equation*}
\psi(b,t)\le \psi(a+b,t)\text{ for all } t\in \Snmu\,.
\end{equation*}
  \end{enumerate}
\end{lemma}
We observe that the construction in \ref{lemmapsiinequalities1} does not work
for $n\ge 10$. 
Indeed, if we take $n=10$, $\eta=1$, $b=2e_1+\sum_{i=2}^{10} e_i$,
$t=\frac12 e_1 - (12)^{-1/2} \sum_{i=2}^{10}e_i$  then a short computation shows
that $\psi(b,t)<\psi(b-e_1,t)+\psi(e_1,t)$.
\begin{proof}
\ref{lemmapsiinequalities1}: 
We need to choose $a$ such that the quantity
\begin{alignat*}1
\xi&=\psi(b,t)-\psi(a,t)-\psi(b-a,t)
=2(b-a)\cdot a + 2 \eta ((b-a)\cdot t)(a\cdot t)
\end{alignat*}
is nonnegative.
We set
\begin{equation*}
  a=\sum_{i: |b_i|=\beta} \sgn(b_i) e_i,
\end{equation*}
so that  $\max_i |a_i|=1$, $\max_i |b_i-a_i|=\beta-1$,
$b=\beta a + b'$ and $a\cdot b'=0$.
Then
\begin{alignat*}1
  \xi&=2 (\beta-1)|a|^2 + 2 \eta (\beta-1) (a\cdot t)^2 + 2\eta
  (b'\cdot t) (a\cdot t)\\
  &\ge 2 (\beta-1)|a|^2 \eta \left[ 1 + x^2 - \frac{|b'|}{|a|(\beta-1)} x
    \sqrt{1-x^2} \right],
\end{alignat*}
where we set $x=|a\cdot t|/|a|$ and used that, since $a$ and $b'$ are
orthogonal, $|b'\cdot t|\le |b'|\sqrt{1-x^2}$. Since $b'$ has at most $n-1$
non-zero 
components, each of them has length at most $\beta-1$, and $|a|\ge 1$ we have
$\frac{|b'|}{|a|(\beta-1)} \le \sqrt{n-1}\le \sqrt8 =2\sqrt2$. The
conclusion follows from the fact that 
$2\sqrt2 x\sqrt{1-x^2}\le (\sqrt 2 x)^2+ (1-x^2)=1+x^2$.

\ref{lemmapsiinequalities1largen}: We set $a=\sum_i {\rm sgn}(b_i)\lceil |b_i|/2\rceil e_i$,
$f=b-2a$, and compute, with $\xi$ as above,
\begin{alignat*}1
\xi=&2\bigl(|a|^2+f\cdot a + \eta (a\cdot t)^2+\eta (a\cdot t)(f\cdot t)\bigr)
\ge 2 (|a|^2-2|a|\, |f|)\,.
\end{alignat*}
The conclusion follows from  $|f|\le \sqrt n$ and $|a|\ge |b|/2\ge 2\sqrt n$.

\ref{lemmapsiinequalities2}:
We write
\begin{alignat*}1
 \psi(a+b,t)-  \psi(b,t)
&= |a|^2+|b|^2+\eta(t\cdot a+t\cdot b)^2 - (|b|^2+\eta (t\cdot b)^2)\\
&= |a|^2+\eta[(t\cdot a+t\cdot b)^2 -  (t\cdot b)^2]\\
&\ge\eta[|a|^2 + (t\cdot a)^2+ 2 (t\cdot a)(t\cdot b)]\,.
\end{alignat*}
As in the previous case we set
 $x=|a\cdot t|/|a|$ and use orthogonality to write
\begin{alignat*}1
 \psi(a+b,t)-  \psi(b,t)
&\ge\eta |a|^2[1+x^2-\frac{2|b|}{|a|} x\sqrt{1-x^2}]\,.
\end{alignat*}
The conclusion follows, using
$|b|\le |a|\sqrt2$, with the same inequality as 
in \ref{lemmapsiinequalities1}.
\end{proof}

\subsection{Explicit relaxation for special $b$}
\label{secexplicit}
\begin{lemma}\label{lemmabbetaeu}
For $n\le 9$ and all $i\in \{1,\dots, n\}$, $\beta\in\Z$ we have
\begin{equation*}
\bar\psi(\beta
e_i,t)=|\beta|\psi(e_i,t) \,.
\end{equation*}
\end{lemma}
\begin{proof}
  The inequality  $\bar\psi(\beta e_i,t)\le |\beta|\psi(e_i,t)$ follows from subadditivity. To
  prove the converse inequality, we first observe that
  \begin{equation*}
    \psi(e_i,t)\le \psi(\alpha,t) \text { whenever } \alpha_i\in\{-1,1\}\,.
  \end{equation*}
Indeed, it suffices to apply  Lemma
\ref{lemmapsiinequalities}\ref{lemmapsiinequalities2} with $b=\alpha_i e_i$,
and $a=\alpha-b$, which is admissible because $|b|=1$ and $|a|\ge 1$ (unless
$a=0$, but in this case there is nothing to prove).

Let $T$ be a minimizer in the lower bound (\ref{eqlbtalpha}). We estimate,
using the above observation and then
Lemma \ref{lemmatildepsisubadditive},
\begin{equation*}
  \sum_{\alpha: \alpha_i\ne 0} \psie(\alpha, T_\alpha) \ge
  \sum_{\alpha: \alpha_i\ne 0} \psie(e_i, T_\alpha) \ge
  \psie(e_i, \sum_{\alpha: \alpha_i\ne 0} \alpha_i T_\alpha)=
  \psie(e_i, z)\,,
\end{equation*}
where we defined 
 $z=\sum_{\alpha: \alpha_i\ne 0} \alpha_i T_\alpha$.
The $i$-th row of the  condition $\sum_\alpha \alpha\otimes T_\alpha=b\otimes
t$ gives then $z=\beta t$.  We conclude that
\begin{equation*}
  \bar\psi(\beta e_i,t)\ge \psie(e_i,\beta t)=|\beta| \psi(e_i,t)
\end{equation*}
and therefore the statement.
\end{proof}

\begin{lemma}\label{lemmaeiej}
For $n\le 9$ and all $\beta\in \Z$, $t\in \Snmu$, $i\ne j\in \{1,\dots, n\}$ we have
  \begin{alignat*}1
    \bar\psi(\beta(e_i+e_j),t)=&|\beta|\min \left\{
      \psie(e_i,z_1)+\psie(e_j,z_2)
\phantom{\psie(e_i-e_j,\frac{z_2-z_1}{2})}\right.\\
&\left.+\psie(e_i-e_j,\frac{z_2-z_1}{2})
+\psie(e_i+e_j,t-\frac{z_1+z_2}2): z_1,z_2\in
      \R^n\right\}
  \end{alignat*}
and correspondingly for $\beta(e_i-e_j)$.
\end{lemma}
\begin{proof}
  \begin{steplist}
    \item Lower bound. 
For ease of notation we focus on the  case $i=1$, $j=2$. Let $T$ be a minimizer in the lower bound (\ref{eqlbtalpha}) corresponding to $\beta(e_1+e_2)$. We define
\begin{alignat*}2
  T_1=&\sum_{\alpha_1\ne 0, \alpha_2=0} \alpha_1 T_\alpha\,,
&\hskip1cm &
  T_2=\sum_{\alpha_1= 0, \alpha_2\ne 0} \alpha_2 T_\alpha\,,\\
  T_+=&\sum_{\alpha_1=\alpha_2 \ne 0} \alpha_1 T_\alpha \,,&&
  T_-=\sum_{\alpha_1=-\alpha_2 \ne 0} \alpha_1 T_\alpha\,.
\end{alignat*}
The sets over which these sums run are disjoint, and $\alpha_1=\alpha_2=0$ on
all other values of $\alpha$. Therefore the first two rows of
$\sum_\alpha \alpha\otimes T_\alpha=\beta(e_1+e_2)\otimes t$ give
\begin{alignat}1\label{eqt1t2tpbetat}
  T_1+T_++T_-&=\beta t \hskip3mm\text{ and }\hskip3mm
  T_2+T_+-T_-=\beta t \,.
\end{alignat}
In particular, $ T_1-T_2+2T_-=0$. 
We decompose the sum of the $\psi(\alpha, T_\alpha)$ in (\ref{eqlbtalpha})
into the same four parts as above.

Let us start with the  part with $\alpha_1=\alpha_2\ne0$. For each $\alpha$
with this property 
we consider $b=\alpha_1(e_1+e_2)$ and $a=\alpha-b$. Then
$a\cdot b=0$ and, recalling that $|\alpha_1|=1$, we have $\sqrt2= |b|\le |a|\sqrt2$ (unless $a=0$, but in this case there is nothing
to prove!). By Lemma \ref{lemmapsiinequalities}\ref{lemmapsiinequalities2} we
obtain $\psie(e_1+e_2,t)\le \psie(\alpha, t)$ for all $t$.  Therefore
\begin{equation*}
  \sum_{\alpha_1=\alpha_2\ne 0} \psie(\alpha, T_\alpha)
  \ge   \sum_{\alpha_1=\alpha_2\ne 0} \psie(e_1+e_2, \alpha_1 T_\alpha)
  \ge    \psie(e_1+e_2, T_+)\,,
\end{equation*}
where in the last step we used 
the subadditivity of Lemma \ref{lemmatildepsisubadditive}.
The case $\alpha_1\ne 0=\alpha_2$ is  similar and
has already been treated in the  proof of Lemma \ref{lemmabbetaeu},
\begin{equation*}
  \sum_{\alpha_1\ne0,\alpha_2=0} \psie(\alpha, T_\alpha)
  \ge   \sum_{\alpha_1\ne0,\alpha_2=0} \psie(e_1, \alpha_1 T_\alpha)
  \ge    \psie(e_1, T_1)\,.
\end{equation*}
The other two cases are almost identical. 
Therefore we have shown that 
\begin{alignat*}1
\bar  \psi(\beta(e_1+e_2),t)\ge&
    \ \psie(e_1, T_1)+
    \psie(e_2, T_2)+
    \psie(e_1+e_2, T_+)+
    \psie(e_1-e_2, T_-)\,.
\end{alignat*}
We  set $z_1=T_1/\beta$, $z_2=T_2/\beta$. By (\ref{eqt1t2tpbetat}) 
one has $T_-=\beta (z_2-z_1)/2$ and $T_+=\beta ( t-(z_1+z_2)/2)$. Since
$\psie$ is positively 1-homogeneous in the second argument,
\begin{alignat*}1
\bar  \psi(\beta(e_1+e_2),t)\ge&
\ |\beta| \psie(e_1, z_1)+
|\beta| \psie(e_2, z_2)\\
&+
|\beta| \psie(e_1+e_2, t-\frac{z_1+z_2}2)+
|\beta|    \psie(e_1-e_2, \frac{z_2-z_1}2)\,.
\end{alignat*}
\item Upper bound. It suffices to consider $\beta=1$, the other cases follow by
  subadditivity. 
The construction is illustrated in Figure \ref{fig:constr}. 
Precisely, we let $\gamma_1$ be the polygonal curve that joins (in this order)
the points
\begin{equation*}
  (0,0), \hskip2mm
  \frac12 z_1, \hskip2mm
  \frac12 z_2, \hskip2mm
  \frac12 (z_1+z_2), \hskip2mm
  t\,,  
\end{equation*}
and $\tau_1$ its tangent vector. Analogously, let $\gamma_2$ be the curve that
joins
\begin{equation*}
  (0,0), \hskip2mm
  \frac12 z_2, \hskip2mm
  \frac12 z_1, \hskip2mm
  \frac12 (z_1+z_2), \hskip2mm
  t\,,  
\end{equation*}
and $\tau_2$ its tangent. Then we set
\begin{equation*}
  \mu= e_1 \otimes \tau_1\calH^1\LL \gamma_1 + 
e_2 \otimes \tau_2\calH^1\LL \gamma_2 \,.
\end{equation*}
One can then extend $\mu$ $t$-periodic and rescale to get a sequence
$\mu_k\to(e_1+e_2)\otimes t \calH^1\LL(\R t)$ and prove the upper bound.
  \end{steplist}
\end{proof}

\begin{figure}[t]
\begin{center}
 \includegraphics[width=\textwidth]{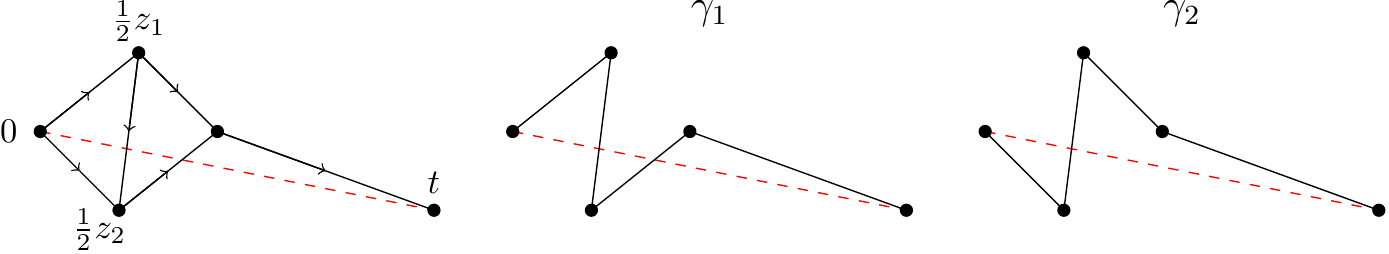}
\end{center}
\caption{Sketch of the construction used in the upper bound of Lemma
  \ref{lemmaeiej}.
  The left panel shows the support of the measure, the central one the part on
  which $\alpha_1\ne0$, the right one the part on which $\alpha_2\ne 0$. The
  red dashed line is $t$.}
\label{fig:constr} 
\end{figure}

The following, more explicit result in two dimensions was mentioned without
proof  in \cite{Con_Gar_Muel}. It shows that in this case the relaxation is
obtained first by making the integrand subadditive in the first argument than
taking the 
(one-homogeneous) convex envelope in the second argument of the result,
corresponding to the upper bound given in \cite{CacaceGarroni2009}. 
In particular, the minimum is not always trivial. For example, for $t=e_2$ it
is easy to see that whenever $\eta>0$ the minimizer obeys $z\cdot e_1>0$.
The resulting microre is illustrated in Figure \ref{fig:example}.

\begin{figure}[t]
\begin{center}
 \includegraphics[width=\textwidth]{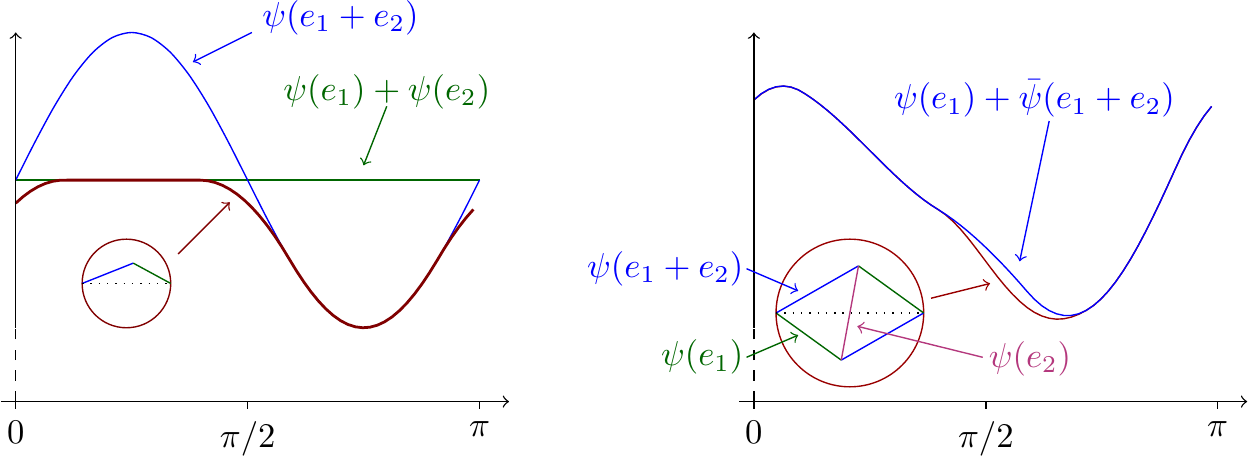}
\end{center}
\caption{Left panel: $\bar\psi(e_1+e_2,t)$ as given in
  Lemma   \ref{lemmabarpsi11} as a function of $\alpha$, for $\eta=1$, $t=(\cos\alpha,\sin\alpha)$. The two 
  one-dimensional options   $\psi(e_1,t)+\psi(e_2,t)=2+\eta$ and
  $\psi(e_1+e_2,t)=2+\eta(1+ t_1t_2)$ are   optimal for different
  orientations. Close to the intersection a mixture of the two
  options is optimal, as sketched in the inset. 
  Right panel: Corresponding plot for $\bar\psi(2e_1+e_2,t)$ (different
  vertical scale). For most values
  of $t$ the optimal energy is obtained using
  $\psi(e_1,t)+\bar\psi(e_1+e_2,t)$. The latter is the convex, subadditive envelope of $\psi$, see discussion at the end of Section 
\ref{secexplicit}. However, there is a region in which a
  more complex  structure develops (sketched in the inset), leading to a 
lower energy. The latter
  construction   bears similarity
  to the examples given in  \cite{AB2,Caraballo2009}.}
\label{fig:example} 
\end{figure}

\begin{lemma}\label{lemmabarpsi11}
For $n=2$ and all $\beta\in \Z$, $t\in \Suno$ we
have 
    \begin{equation*}
    \bar\psi(\beta(e_1+e_2),t)=|\beta|\min\left\{
      \psie(e_1,z)+\psie(e_2,z)+\psie(e_1+e_2,t-z): z\in
      \R^2\right\}\,.
  \end{equation*}
\end{lemma}
\begin{proof}
We just need to show that minimum in the formula of Lemma 
 \ref{lemmaeiej}  is attained at $z_1=z_2$. 
This is equivalent to the statement that
  \begin{alignat*}1
\psie(e_1,m-d)+\psie(e_2,m+d)+\psie(e_1-e_2,d)
    - \psie(e_1,m)-\psie(e_2,m)\ge0
  \end{alignat*}
for all $m,d\in\R^2$ (we set $z_1=m-d$, $z_2=m+d$).
Explicitly, this expression is
  \begin{multline*}
|m-d|+|m+d|+2|d|+\eta\frac{(m_1-d_1)^2}{|m-d|}
    +\eta\frac{(m_2+d_2)^2}{|m+d|}+
    \eta \frac{(d_1-d_2)^2}{|d|} \\
- 2|m|-\eta \frac{m_1^2+m_2^2}{|m|}\,.
      \end{multline*}
      Clearly $|m+d|+|m-d|\ge2|m|$, $(d_1-d_2)^2\ge0$ and $2|d|\ge
      2\eta|d|$. Therefore it   suffices to show that 
  \begin{alignat*}1
\xi=2|d|+\frac{(m_1-d_1)^2}{|m-d|}
    +\frac{(m_2+d_2)^2}{|m+d|}-|m|\ge 0\,
      \end{alignat*}
for all $m,d\in\R^2$.
  We set $m-d=r (\cos\theta,\sin\theta)$, 
  $m+d=s (\cos\varphi,\sin\varphi)$, with $r,s\in(0,\infty)$, $\theta,\phi\in
  \R$. 
 From  $2m=(m+d)+(m-d)$ we obtain $|m|\le (r+s)/2$, and 
 with $2d=(m+d)-(m-d)$ we have $\xi\ge\zeta$, where
  \begin{alignat*}1
    \zeta&=   \sqrt{r^2+s^2-2rs\cos(\varphi-\theta)}+ r \cos^2\theta + s \sin^2\varphi 
   - \frac12
    (r+s)\\
&=  \sqrt{r^2+s^2-2rs\cos(\varphi-\theta)}+\frac12r \cos(2\theta) - \frac12s \cos(2\varphi) 
  \end{alignat*}
since $\frac12\cos2\theta=\cos^2\theta-\frac12=\frac12-\sin^2\theta$.
We change variables again, and write $2\theta=\gamma-\delta$,
$2\varphi=\gamma+\delta$.  
Then 
\begin{alignat*}1
  2\zeta&= r \cos(\gamma-\delta) - s \cos(\gamma+\delta) +
  2\sqrt{r^2+s^2-2rs\cos\delta}\,.
\end{alignat*}
With $\cos(\gamma-\delta)=\cos\gamma\cos\delta+\sin\gamma\sin\delta$ we obtain
\begin{alignat*}1
  2\zeta&= (r-s)\cos\gamma\cos\delta
+(r+s)\sin\gamma\sin\delta+
   2 \sqrt{r^2+s^2-2rs\cos\delta}\,.
\end{alignat*}
The first two terms are the scalar product of $(\cos\gamma,\sin\gamma)$ with
another vector, which  is bounded by the length of the vector. Therefore
\begin{alignat*}1
  2\zeta&\ge
   2 \sqrt{r^2+s^2-2rs\cos\delta}
- \sqrt{(r-s)^2\cos^2\delta
+(r+s)^2\sin^2\delta}\\
&=   2 \sqrt{r^2+s^2-2rs\cos\delta}
- \sqrt{(r+s)^2-4rs\cos^2\delta}\,.
\end{alignat*}
Squaring, the last expression is nonnonegative iff
\begin{alignat*}1
4r^2+4s^2-8rs\cos\delta
\ge (r+s)^2-4rs\cos^2\delta\,,
\end{alignat*}
which in turn is equivalent to
\begin{alignat*}1
3r^2+3s^2-2rs+4rs(\cos^2\delta-2\cos\delta)
\ge 0\,,
\end{alignat*}
which is true since $x^2-2x\ge -1$ and $r^2+s^2\ge 2rs$.
\end{proof}
In closing, we remark that the relaxation for other values of $b$ is more
complex and includes other microstructures. To see this, we 
 define $\psi^*$ by
\begin{equation}
  \label{eq:psicsa}
  \psi^*(b,t)=\min\left\{ \sum_{i=1}^N 
  \bar\psi(z^i,t): N\in\N, z^i\in \{-1,0,1\}^2, 
  \sum_{i=1}^N z^i=b\right\}.
\end{equation}
The values of $\bar\psi$ entering this expression are characterized in Lemma \ref{lemmabbetaeu} and Lemma \ref{lemmabarpsi11}. The function $\psi^*$ is by definition  subadditive in $b$, existence of the minimum follows from growth and continuity.
We now  show that
a sequence $\{z^1,\dots, z^N\}$ which contains
 a pair $(z,z')$ with $z_1=-z'_1=1$ cannot be optimal. 
If $z+z'=0$, it suffices to remove both of them. If $z+z'=\pm e_2$,
replacing the pair by $\pm e_2$ reduces the energy, since
$\bar\psi(e_2)\le \bar\psi(e_1)+\bar\psi(e_1\pm e_2)$. 
If $z+z'=\pm 2e_2$ then replacing the pair with $(\pm e_2,\pm e_2)$ reduces the energy, 
since $2\bar\psi(e_2)\le \bar\psi(e_1+e_2)+\bar\psi(e_1-e_2)$.
Therefore the sign of all $z_1^i$ is the same. Analogously for the $z_2^i$, and
one concludes that 
\begin{alignat*}1
  \label{eq:psicsab}
  \psi^*(b,t)=&\min\{ |b_1|,|b_2|\} \bar\psi(e_1 + \sgn(b_1b_2) e_2 ,t)\\
&  + (|b_2|-|b_1|)_+ \psi(e_2,t)
  + (|b_1|-|b_2|)_+ \psi(e_1,t).
\end{alignat*}
This expression is clearly convex in $t$. Finally, we show that $\psi^*\le\psi$. This is immediate if $|b|\le \sqrt2$, and 
follows from quadratic growth of $\psi$ for larger $b$. In particular, if 
$|b_1|$ and $|b_2|$ are not 1 then from $\psi(e_1,t)\le 2$ we obtain
$\psi^*(b,t)\le 2|b_1|+2|b_2|\le b_1^2+b_2^2 \le \psi(b,t)$. 
If $|b_1|=1$ and $|b_2|\ge 3$, a similar computation holds since
$2|b_1| +2|b_2|\le 1+3|b_2|\le |b|^2$. It remains to deal with the case
$b=(1,2)$ (up to signs and permutations). In this case, from $\eta |2t_1t_2|\le |t|^2=1$ we obtain  
\begin{equation*}
 \psi^*((1,2),t) \le 3 + \eta(t_1^2+2t_2^2) 
 \le 5 + \eta (t_1^2+4t_2^2+4t_1t_2)=\psi((1,2),t)\,.
\end{equation*}
Therefore $\psi^*\le \psi$. We conclude that $\psi^*$ is the convex subadditive envelope of $\psi$. 

In Figure \ref{fig:example} we investigate the case
$b=(2,1)$ in more detail. The lower bound (\ref{eqlbtalpha}) 
is (numerically) attained by a
microstructure in which $\alpha=(1,1)$, $\alpha=(1,0)$ and $\alpha=(0,1)$ play
a role, and is smaller than $\psi^*$. Therefore in this case $\bar\psi<\psi^*$.

\section*{Acknowledgements}
We thank Giovanni Alberti and Camillo De Lellis for fruitful discussions and useful suggestions. 
This work was partially supported by the Deutsche Forschungsgemeinschaft
through the Sonderforschungsbereich 1060 
{\em ``The mathematics of emergent effects''}, project A5.

%\bibliographystyle{amsplain}
%\bibliographystyle{amsplain-initials}
%\bibliography{Bib_Co-Gar-Mas}
\providecommand{\bysame}{\leavevmode\hbox to3em{\hrulefill}\thinspace}
\providecommand{\MR}{\relax\ifhmode\unskip\space\fi MR }
\providecommand{\MRhref}[2]{%
  \href{http://www.ams.org/mathscinet-getitem?mr=#1}{#2}
}
\providecommand{\href}[2]{#2}

\end{document}